\documentclass[authoryear]{elsarticle}
\usepackage[a4paper, total={6in, 10in}]{geometry}
\usepackage{booktabs}
\usepackage{float}
\usepackage{amsmath}
\usepackage{amsthm}
\usepackage{amsmath,amssymb,amsfonts,mathrsfs,hyperref,microtype, amsthm}
\usepackage{graphicx}
\usepackage{graphicx,mathabx}
\usepackage{enumitem}  
\graphicspath{ {images/} }
\numberwithin{equation}{section}

\newtheorem{theorem}{Theorem}[section]
\newtheorem{lemma}{Lemma}[section]
\newtheorem{corollary}{Corollary}[section]
\newtheorem*{remark}{Remark}
\newcommand{\BS}{\boldsymbol}

\newcommand{\rmnum}[1]{\romannumeral #1}
\newcommand{\Rmnum}[1]{\expandafter\@slowromancap\romannumeral #1@}
\newcommand{\trp}{{\sf T}}
\journal{}
\makeatletter
\def\ps@pprintTitle{%
   \let\@oddhead\@empty
   \let\@evenhead\@empty
   \def\@oddfoot{\reset@font\hfil\thepage\hfil}
   \let\@evenfoot\@oddfoot
}

\begin{document}

\begin{frontmatter}
\author[a]{ Osama Idais\corref{cor1}}
\ead{osama.idais@ovgu.de}
\address[a]{\small Institute for Mathematical Stochastics, Otto-von-Guericke University Magdeburg, \\ \small PF 4120,
39016 Magdeburg, Germany }

 \title{Locally optimal designs for generalized linear models \\ within the family of  Kiefer $\Phi_k$-criteria}

\begin{abstract}
Locally optimal designs for generalized linear models are derived at certain values of the regression parameters. In the present paper a general setup of the generalized linear model is considered.  Analytic solutions for optimal designs are developed under Kiefer  $\Phi_k$-criteria highlighting  the D- and A-optimal designs. By means of The General Equivalence Theorem necessary and sufficient conditions in term of intensity values are obtained to characterize the locally optimal designs. In this context, linear predictors  are assumed constituting first order models with and  without intercept on appropriate experimental regions.
\end{abstract}


\end{frontmatter}
\section{Introduction}
\label{}
The generalized linear model (GLM) was developed by  \citet{10.2307/2344614}. It is viewed as a generalization of the ordinary linear regression which allows continuous or discrete observations from  one-parameter exponential family distributions to be combined with explanatory variables (factors) via proper link functions. Therefore,  wide applications  can be addressed by  GLMs such as social and educational sciences, clinical trials, insurance, industry. In particular;  logistic and probit models are used for binary observations whereas  Poisson models and gamma models are used  for count and nonnegative continuous observations, receptively (\citet{10.2307/2333860}, \citet{doi:10.1080/00224065.1997.11979769}, \citet{fox2015applied}, \citet{goldburd2016generalized}).  Methods of likelihood are utilized to obtain the  estimates of the model parameters. The precision of these maximum likelihood estimates (MLEs) is measured by their variance-covaraince matrix. In ordinary regression models for which  normality assumption is realized  the variance-covariance matrix is exactly (proportional to) the inverse of the Fisher information matrix. In contrast, for the GLMs  the observations are often non-normal, and therefore  large sample theory is demanded for the statistical inference. In this context, the variance-covariance matrix is approximately  the inverse of the Fisher information matrix. It should, however, be emphasized that the Fisher information matrix for GLMs  depends on the model parameters. The theory of generalized linear models is presented carefully in \citet{mccullagh1989generalized} and  \citet{dobson2018introduction}. \par
While deriving optimal designs is obtained by minimizing the variance-covariance matrix there is no loss of generality to concentrate on maximizing the Fisher information matrix. For generalized linear models the optimal design cannot be found without a prior knowledge of the parameters (\citet{khuri2006}, \citet{atkinson2015designs}). One approach which so-called local optimality was proposed by \citet{chernoff1953} aiming at  deriving a locally optimal design at a given parameter value (best guess). This approach  is widely employed for GLMs, for instance; for count data with Poisson models and Rasch Poisson model see  \citet{WANG20062831}, \citet{10.2307/24308852} and Graßhoff, Holling, and Schwabe (\citeyear{10.1007/978-3-319-00218-7_14}, \citeyear{10.1007/978-3-319-13881-7_9}, \citeyear{2018arXiv181003893G}). For binary data:  see  \citet{10.2307/2287114} and \citet{MATHEW2001295} under logistic models and \citet{biedermann2006optimal} under dose-response models whereas under logit, log-log and probit models see \citet{10.2307/24310039}. Furthermore, \citet{GAFFKE2019} provided locally D- and A-optimal designs for gamma models.  In particular, optimal designs for GLMs without intercept have not been considered carefully. \citet{doi:10.1080/02331888.2014.937342} provided analytic proofs of  D-optimal designs for zero intercept parameters of a two-binary-factor logistic model with no interaction. Recently, \citet{2019arXiv190409232I} introduced locally D- and A-optimal designs for gamma models without intercept.\par

Locally optimal designs for a general setup of generalized linear models received some attention.  Geometrically,  \citet{10.2307/2346142} considered only one continuous factor. \citet{ATKINSON1996437} presented a study of optimal designs for nonlinear model including GLMs. \citet{YANG2008624} provided optimal designs for GLMs with applications to logistic and probit models. Also a general solution for GLMs was given in  \citet{yang2009}.  Analytic solutions under D-criterion were obtained by \citet{tong2014} for particular limitations. \par    

The paper is organized as follows. In Section \ref{sec3-1} we present some approaches to determine the optimal weights for particular designs under D-, A- and $\Phi_k$-criteria which will be used in the subsequent sections. Throughout, with the aid of The Equivalence Theorem  we establish a necessary and sufficient condition for a design  to be locally D-,  A- or $\Phi_k$-optimal designs. We begin with the single-factor model by Section \ref{sec3-2}. In  Section \ref{sec3-3} we consider  first order models with intercept.  In Section \ref{sec3-5} we focus on Kiefer $\Phi_k$-criteria for first order models without intercept.
\section{Preliminary}\label{sec-2}

In the following subsections we introduce  the GLMs and the required notations of optimal design theory. 

\subsection{Model specification} \label{subsec2-1}
In the context of the generalized linear models the observations (responses) belong to a  one-parameter exponential family. The probability density function of a response variable $Y$ defined as
\begin{equation}
p(Y;\theta,\phi)=\exp\Big(\frac{Y\theta-b(\theta)}{a(\phi)}+c(\phi,Y)\Big), \label{eq-exp}
\end{equation}
where $a(\cdot),b(\cdot)$ and $c(\cdot)$ are known functions whereas $\theta$ is a canonical parameter  and $\phi$ is a dispersion parameter. 
 A common computational method for fitting the models to data are provided in the GLM framework. That is the expected mean is given by  $E(Y)=\mu=b^\prime(\theta)$, and the variance is given by  $\mathrm{var}(Y)=a(\phi)b^{\prime\prime}(\theta)$. The quantity $b^{\prime\prime}(\theta)$ is called the mean-variance function or equivalently, the variance function of the expected  mean, i.e., $V(\mu)=b^{\prime\prime}(\theta)$. Thus we may write $\mathrm{var}(Y)=a(\phi)V(\mu)$ which depends on the values of $\BS{x}$  (see \citet{mccullagh1989generalized}, Section 2.2.2).\par
Consider the experimental region ${\cal X}\subseteq  \mathbb{R}^{\nu}, \nu\ge1,$ to which the covariate value $\BS{x}$  belongs. Denote by $\BS{\beta}\in \mathbb{R}^{p}$ the parameter vector. Let $\BS{f}(\BS{x}):{\cal X}\rightarrow \mathbb{R}^{p}$ be a $p$-dimensional regression function, i.e.,  $\BS{f}(\BS{x})=(f_{1},\dots, f_{p})^\trp$ where the components $f_{1},\dots, f_{p}$ are real-valued continuous linearly independent functions. The generalized linear model can be introduced as  
 \begin{equation}
\eta=g(\mu)\,\,  \mbox{ where } \eta=\BS{f}^\trp(\BS{x})\BS{\beta}, \label{eq2.2}
\end{equation}
where  $g$  is a link function that relates the expected mean $\mu$ to the linear predictor $\BS{f}^\trp(\BS{x})\BS{\beta}$. It is assumed that $g$ is one-to-one and differentiable. 
One can realize that $\mu=\mu(\BS{x},\BS{\beta})=g^{-1}\big(\BS{f}^\trp(\BS{x})\BS{\beta}\big)$ and $\mathrm{d} \eta/\mathrm{d}\mu=g^{\prime}\bigl(g^{-1}\big(\BS{f}^\trp(\BS{x})\BS{\beta}\big) \bigr)$ and therefore, we can define the  intensity function at a point $\BS{x} \in \mathcal{X}$ as
\begin{equation}
 u(\BS{x},\BS{\beta})=\Bigl(\mathrm{var}(Y)\,\Big(\frac{{\rm d}\eta}{{\rm d}\mu}\Big)^2\Bigr)^{-1} \label{eq2.3}
  \end{equation}
 which is positive and depends on the value of linear predictor $\BS{f}^\trp(\BS{x})\BS{\beta}$. The intensity function is regarded as the weight for the corresponding unit at the point $\BS{x}$  (\citet{atkinson2015designs}).\par
The Fisher information matrix for a GLM at $\BS{x}\in \mathcal{X}$ (see \citet{fedorov2013optimal}, Subsection 1.3.2) has the form
\begin{equation}
 \BS{M}(\BS{x},\BS{\beta})=u(\BS{x},\BS{\beta})\,\BS{f}(\BS{x})\,\BS{f}^\trp(\BS{x}).  \label{eq2.4}
 \end{equation}
For the whole experimental points $\BS{x}_1, \dots, \BS{x}_n$ the Fisher information matrix  reads as
\begin{equation}
\BS{M}(\BS{x}_1,\dots,\BS{x}_n,\BS{\beta})=\sum_{i=1}^{n}\BS{M}(\BS{x}_i,\BS{\beta}). \label{eq-info}
\end{equation}
The information matrix of the form (\ref{eq2.4}) is appropriate for other nonlinear models, e.g., The survival times observations which depend on the proportional hazard model (\citet{schmidt2017optimal}). Moreover, under homoscedastic regression models the intensity function is constant equal to $1$ whereas, under heteroscedastic regression models we get intensity that is equal to $1/\mathrm{var}(Y)$ which depends on $\BS{x}$ only and thus  we have information matrix of form $ \BS{M}(\BS{x})=u(\BS{x})\,\BS{f}(\BS{x})\,\BS{f}^\trp(\BS{x})$ that does not depend on the model parameters. The latter case was discussed  in  \citet{GRAHOFF20073882} and in the book by  \citet{fedorov2013optimal}, p.13.  

It is worthwhile mentioning that unlike the normally-distributed response variables, the sampling distributions for MLEs\, $\BS{\hat\beta}$ in GLMs that used for inference cannot be determined exactly. Therefore,  the statistical inferences for GLMs are conducted for large sample sizes under mild regularity assumptions on the probability density (\ref{eq-exp}). Hence,  
\[
\sqrt{n}(\BS{\hat\beta}_{n}-\BS{\beta}) \xrightarrow[]{\text{d}}
 \mathcal{N}_{p}\big(\BS{0},\BS{M}^{-1}\big)
 \]
where $\BS{M}=\lim\limits_{n \to \infty}\frac{1}{n}\BS{M}(\BS{x}_1,\dots,\BS{x}_n,\BS{\beta})$ (\citet{fahrmeir1985}, Theorem 3). 
Moreover,  the variance-covariance matrix of $\BS{\hat\beta}$ is approximately given by  the inverse of the Fisher information matrix (\ref{eq-info}), see \citet{fedorov2013optimal}, Section 1.5.,  
\begin{equation}
\mathrm{var}(\BS{\hat\beta})\approx\BS{M}^{-1}(\BS{x}_1,\dots,\BS{x}_n,\BS{\beta}). \label{eq-var}
\end{equation}
\subsection{Optimal designs} \label{subsec2-2}
Throughout the present work we will deal with the approximate (continuous) design theory, i.e., a design $\xi$ is a probability measure with finite support on 
the experimental region ${\cal X}$, 
\begin{equation}
\xi=\left( \begin{array}{cccc}   \BS{x}_1 &\BS{x}_2&\dots&\BS{x}_r  \\  
 \omega_1 & \omega_2 &\dots&\omega_r \end{array}\right), \label{eq2-5}
\end{equation}
where $r\in\mathbb{N}$, $\BS{x}_1,\BS{x}_2, \dots,\BS{x}_r\in\mathcal{X}$ are pairwise distinct points  
and $\omega_1, \omega_2, \dots, \omega_r>0$ with $\sum_{i=1}^{r} \omega_i=1$. 
The set ${\rm supp}(\xi)=\{\BS{x}_1,\BS{x}_2, \dots,\BS{x}_r\}$ is called the support of $\xi$ and 
$\omega_1,\ldots,\omega_r$ are called the weights of $\xi$,
see \citet{silvey1980optimal}, p.15.  
The information matrix of a design $\xi$ from (\ref{eq2-5}) at a parameter 
point $\BS{\beta}$ is defined by
\begin{equation}
\BS{M}(\xi, \BS{\beta})=\int_{\mathcal{X}} \BS{M}(\BS{x}, \BS{\beta})\, \xi(\mathrm{d} \BS{x})=  \sum_{i=1}^{r}\omega_i \BS{M}(\BS{x}_i, \BS{\beta}).\label{eq2-6}
\end{equation}
One might recognize $\BS{M}(\xi, \BS{\beta})$ as a convex combination of all information matrices for all design points of $\xi$. Another representation of the information matrix (\ref{eq2-6}) can be utilized  based on  the $r \times p$ design matrix $\BS{F}=[\BS{f}(\BS{x}_1),\dots,\BS{f}(\BS{x}_r)]^\trp$ and the $r\times r$ weight matrix $\BS{V}=\mathrm{diag}(\omega_iu(\BS{x}_i,\BS{\beta}))_{i=1}^{r}$ and hence we can write 
\[
\BS{M}(\xi, \BS{\beta})=\BS{F}^\trp\BS{V}\BS{F}.
\]

\begin{remark} A particular type of designs appears frequently when the support size equals the dimension of $\BS{f}$, i.e.,  $r=p$. In such a case the design is minimally supported and it is often called a minimal-support or a saturated design.  
\end{remark}
In this paper we focus on  optimal designs within the family of Kiefer $\Phi_k$-criteria (\cite{doi:10.1093biomet62.2.277}).  Kiefer $\Phi_k$-criteria aim at minimizing the $k$-norm of the eigenvalues of the variance-covariance matrix and  include the most common criteria D-, A- and E- optimality.  Denote by $\lambda_i(\xi,\BS{\beta})\, (1\le i \le p)$ the eigenvalues  of a nonsingular information matrix $\BS{M}(\xi,\BS{\beta})$. Denote by ``$\det$'' and ``${\rm tr}$'' the determinant and the trace of a matrix, respectively. The Kiefer $\Phi_k$-criteria are defined by 
\begin{align*}
\Phi_k(\xi,\BS{\beta})&=\Big(\frac{1}{p}\mathrm{tr}\big( \BS{M}^{-k}(\xi,\BS{\beta}) \big)\Big)^{\frac{1}{k}}=\Big(\frac{1}{p}\sum_{i=1}^{p}\lambda_i^{-k}(\xi,\BS{\beta})\Big)^{\frac{1}{k}},\,\, 0< k<\infty,\\
\Phi_0(\xi,\BS{\beta})&=\lim_{k\to0+} \Phi_k(\xi,\BS{\beta})=\Big(\det(\BS{M}^{-1}(\xi,\BS{\beta}))\Big)^{\frac{1}{p}},\\
\Phi_\infty(\xi,\BS{\beta})&=\lim_{k\to\infty} \Phi_k(\xi,\BS{\beta})=\max_{1\le i \le p}\big(\lambda_i^{-1}(\xi,\BS{\beta})\big).
\end{align*}
Note that $\Phi_0(\xi,\BS{\beta})$, $\Phi_1(\xi,\BS{\beta})$ and $\Phi_\infty(\xi,\BS{\beta})$ are the D-, A- and E-criteria, respectively. A $\Phi_k$-optimal design $\xi^*$ minimizes  the function $\Phi_k(\xi,\BS{\beta})$ over all designs $\xi$ whose information matrix $\BS{M}(\xi,\BS{\beta})$ is nonsingular. For $0\le k<\infty$ the  strict convexity  of $\Phi_k(\xi,\BS{\beta})$ implies that the information matrix of a locally $\Phi$-optimal design (at $\BS{\beta}$) is unique. That is, if
$\xi^*$ and $\xi^{**}$ are two locally $\Phi$-optimal designs (at $\BS{\beta}$) then $\BS{M}(\xi^*,\BS{\beta})=\BS{M}(\xi^{**},\BS{\beta})$ (\cite{doi:10.1093biomet62.2.277}).  In particular, D-optimal designs are constructed to minimize the determinant of the variance-covariance matrix of the estimates or  equivalently to maximize the determinant of the information matrix. The D-criterion is typically defined by the convex function  $\Phi_{\mathrm{D}}(\BS{M}(\xi, \BS{\beta}))=-\log \det \big(\BS{M}(\xi, \BS{\beta})\big)$.  Geometrically, the volume of the asymptotic  confidence ellipsoid  is inversely proportional to $\sqrt{\det \big(\BS{M}(\xi, \BS{\beta})\big)}$ where $\det \big(\BS{M}(\xi, \BS{\beta})\big)$ can be determined by the inverse of the product of the squared lengths of the axes. Therefore,  the D-optimal designs minimize the volume of the asymptotic confidence ellipsoid.\par
A-optimal designs are constructed to minimize the trace of the variance-covariance matrix of the estimates, i.e., to minimize the average variance of the estimates. The A-criterion is typically defined by $\Phi_{\mathrm{A}}\big(\BS{M}(\xi, \BS{\beta})\big)={\rm tr}\bigl(\BS{M}^{-1}(\xi, \BS{\beta})\bigr)$.  The A-criterion minimizes the sum of the squared lengths of the axes of the asymptotic  confidence ellipsoid. Moreover, E-optimal designs maximize the smallest eigenvalue of $\BS{M}(\xi, \BS{\beta})$ and equivalently, they minimize the squared length of the ‘largest’ axis of the  asymptotic  confidence ellipsoid. 
\par

In order to verify the  local optimality of  a design The General Equivalence Theorem is usually employed  (see \cite{silvey1980optimal}, p.54 and \cite{atkinson2007optimum}, p.137). It  provides necessary and sufficient conditions for a design to be optimal and thus the optimality of a suggested design can be easily verified or disproved. The most  generic  one is the celebrated  Kiefer-Wolfowitz equivalence theorem under  D-criterion (\citet{kiefer_wolfowitz_1960}).  The design  $\xi^*$ is $\Phi_k$-optimal if and only if 
\begin{align}
u(\BS{x},\BS{\beta})\BS{f}^\trp(\BS{x})\BS{M}^{-k-1}(\xi^*,\BS{\beta})\BS{f}(\BS{x})\le \mathrm{tr}(\BS{M}^{-k}(\xi^*,\BS{\beta}))\,\,\,\mbox{for all}\,\, \BS{x}\in \mathcal{X}. \label{eq-3.0}
\end{align}
Furthermore, if the design $\xi^*$ is  $\Phi_k$-optimal then inequality (\ref{eq-3.0}) becomes equality at its support.\par

\begin{remark}\label{rem2.3.33.} \ \\
 The left hand side of  condition (\ref{eq-3.0}) of The General Equivalence Theorem is called the sensitivity function.
\end{remark}
\section{Determination of locally optimal weights} \label{sec3-1}

In this section we provide the optimal weights of the designs that will be derived throughout the paper with respect to Kiefer $\Phi_k$-criteria, and in particular the A-criterion ($k=1$) and the D-criterion ($k=0$).   
In the current work we mostly deal with saturated designs (i.e., $r=p$) for generalized linear models.    Let the support points are given by $\BS{x}_1^*,\dots,\BS{x}_p^*$ such that $\BS{f}(\BS{x}_1^*),\dots, \BS{f}(\BS{x}_p^*)$ are linearly independent. \par

For the A-criterion ($k=1$) the optimal weights are given according to \citet{pukelsheim2006optimal}, Section 8.8, which has been modified in \citet{GAFFKE2019}. The design $\xi^*$ which achieves the minimum value of  \,${\rm tr}\bigl(\BS{M}^{-1}(\xi,\BS{\beta})\bigr)$ over all designs $\xi$ with ${{\rm supp}(\xi^*)=\{\BS{x}_1^*,\ldots,\BS{x}_p^*\}}$ is given by 
\[
\xi^*=\left(\begin{array}{ccc}\BS{x}_1^* & \ldots & \BS{x}_p^*\\ \omega_1^* & \ldots & \omega_p^*\end{array}\right),\ 
\mbox{ with }\ \omega_i^*=c^{-1}\Bigl(\frac{c_{ii}}{u_i}\Bigr)^{1/2}\ (1\le i\le p)\,,\ \
c=\sum\limits_{k=1}^p\Bigl(\frac{c_{kk}}{u_k}\Bigr)^{1/2},
\]
where $u_i=u(\BS{x}_i^*,\BS{\beta})$ ($1\le i\le p$) and $c_{ii}$ ($1\le i\le p$) are the diagonal entries of the matrix 
$\BS{C}=(\BS{F}^{-1})^\trp\BS{F}^{-1}$ and
$\BS{F}=\bigl[\BS{f}(\BS{x}_1^*),\ldots,\BS{f}(\BS{x}_p^*)\bigr]^\trp$.

For the D-criterion ($k=0$) the optimal weights are given by $\omega_i^*=1/p$ ($1\le i\le p$), see Lemma 5.3.1 of \citet{silvey1980optimal}. That is the locally D-optimal saturated design assigns equal weights to the support points.   On the other hand, there is no unified formulas for the optimal weights of a non-saturated design specifically, with respect to D-criterion. However, let the model be given with parameter vector $\BS{\beta}$ of dimension $p=3$, i.e., $\BS{\beta}\in \mathbb{R}^3$. The  next lemma  provides the optimal weights of a design with  four support points   $\xi^*=\{(\BS{x}_i^*,\omega_i^*), i=1,2,3,4 \}$  under certain conditions.   

 \begin{lemma}\label{lem3.0.2.}
Let $\BS{x}_1^*,\,\BS{x}_2^*,\,\BS{x}_3^*,\, \BS{x}_4^* \in \mathcal{X}$ be given such that the vectors $\BS{f}(\BS{x}_1^*)$,\,$\BS{f}(\BS{x}_2^*)$,\,$\BS{f}(\BS{x}_3^*)$,\,$\BS{f}(\BS{x}_4^*)$ are linearly independent. For a given parameter point  $\BS{\beta}$ let $u_{i}=u(\BS{x}_i^*,\BS{\beta})$\,for all \,$(1\le i \le 4)$. Denote 
\begin{align*}
&d_{1}=\det\big[\BS{f}(\BS{x}_{2}^*),\BS{f}(\BS{x}_{3}^*),\BS{f}(\BS{x}_{4}^*)\big],\,\,\, d_{2}=\det\big[\BS{f}(\BS{x}_{1}^*),\BS{f}(\BS{x}_{3}^*),\BS{f}(\BS{x}_{4}^*)\big], \\
&d_{3}=\det\big[\BS{f}(\BS{x}_{1}^*),\BS{f}(\BS{x}_{2}^*),\BS{f}(\BS{x}_{4}^*)\big], \,\,\,
 d_{4}=\det\big[\BS{f}(\BS{x}_{1}^*),\BS{f}(\BS{x}_{2}^*),\BS{f}(\BS{x}_{3}^*)\big]
 \end{align*}
such that $d_i\neq 0$\,\,for all \,$(1\le i\le 4)$. Assume that $u_2=u_3$ and $d_2^2=d_3^2$. Then the design $\xi^*$ which achieves the minimum value of  $-\log\det\bigl(\BS{M}(\xi,\BS{\beta})\bigr)$ over all designs $\xi$ with ${\rm supp}(\xi^*)=\{\BS{x}_1^*, \BS{x}_2^*, \BS{x}_3^*, \BS{x}_4^*\}$ is given by  $\xi^*=\{(\BS{x}_i^*,\omega_i^*),  i=1,2,3,4 \}$ where 
 \begin{align*}
\omega_{1}^*&=\frac{3}{8}+\frac{1}{4}\Big(1+\frac{d_{1}^2}{d_{4}^2}\frac{u_{1}}{u_{4}}-4\frac{d_{2}^2}{d_{4}^2}\frac{u_{1}}{u_{2}}\Big)^{-1},\\
\omega_{2}^*&=\omega_{3}^*=\frac{1}{2}\Big(4-\frac{d_{4}^2}{d_{2}^2}\frac{u_{2}}{u_{1}}-\frac{d_{1}^2}{d_{2}^2}\frac{u_{2}}{u_{4}}\Big)^{-1},\\
\omega_{4}^*&=\frac{3}{8}+\frac{1}{4}\Big(1+\frac{d_{4}^2}{d_{1}^2}\frac{u_{4}}{u_{1}}-4\frac{d_{2}^2}{d_{1}^2}\frac{u_{4}}{u_{2}}\Big)^{-1}.
 \end{align*}
 \end{lemma}

\begin{proof}  Let $\BS{f}_\ell=\BS{f}(\BS{x}_\ell^*)=\big(f_{\ell1},f_{\ell2},f_{\ell3}\big)^\trp\,\,(1\le \ell \le 4)$. The $4 \times 3$ design matrix is given by $ \BS{F}=\bigl[\BS{f}_1,\BS{f}_2,\BS{f}_3,\BS{f}_4\bigr]^\trp$. Denote  $\BS{V}={\rm diag}\bigl(\omega_\ell u_\ell\bigr)_{\ell=1}^{4}$. Then  $\BS{M}(\xi,\BS{\beta})=\BS{F}^\trp\BS{V}\BS{F}$ and by  the Cauchy-Binet formula the determinant of $\BS{M}(\xi,\BS{\beta})$ is given by  the function $\varphi(\omega_1,\omega_2,\omega_3,\omega_4)$ where
\begin{eqnarray}
\varphi(\omega_1,\omega_2,\omega_3,\omega_4)= \sum\limits_{ \substack{ 1\le i<j<k\le4\\ h\in\{1,
2,3,4\}\setminus \{i,j,k\}}} d_{h}^2u_iu_ju_k\,\omega_i\omega_j\omega_k.\label{eq3.1}
\end{eqnarray}
By assumptions $u_2=u_3$, $d_2^2=d_3^2$ the function $\varphi(\omega_1,\omega_2,\omega_3,\omega_4)$ is invariant w.r.t.  permuting $\omega_2$ and $\omega_3$, i.e., $\varphi(\omega_1,\omega_2,\omega_3,\omega_4)=\varphi(\omega_1,\omega_3,\omega_2,\omega_4)$ and thus  minimizing (\ref{eq3.1}) has similar solutions for $\omega_2$ and $\omega_3$. Thus we can write $\omega_4=1-\omega_{1}-2\omega_{2}\,$ then  (\ref{eq3.1}) reduces to
\[
\varphi(\omega_1,\omega_2)= \alpha_1 \omega_2^3+\alpha_2 \omega_2^2+\alpha_3 \omega_1^2 \omega_2+\alpha_4 \omega_2^2 \omega_1+\alpha_5 \omega_1 \omega_2,
\]
where $\alpha_1=-2\,\alpha_2=-2\,d_4^2\,u_2^2\,u_4$, $\alpha_3=-\alpha_5=-4\,d_2^2\,u_1\,u_2\,u_4$, $\alpha_4=u_2^2 \left(d_1^2\,u_1-d_4^2\,u_4\right)-4\,d_2^2\,u_1\, u_2\,u_4$.
 Thus we obtain the system of  two equations $\partial \varphi/\partial \omega_1=0$, ${\partial \varphi/\partial \omega_2=0}$. Straightforward  computations  show that the solution of the above system is  the optimal weights  $\omega_\ell^*\,\,(1\le \ell\le 4)$ presented by the lemma. Hence, these optimal weights  minimizing $\varphi(\omega_1,\omega_2)$. 
\end{proof}

Moreover, saturated designs under Kiefer  $\Phi_k$-criteria for a GLM without intercept are of our interest,  in specific,  under the first order model $\BS{f}(\BS{x})=(x_1,\dots,x_\nu)^\trp$ and a parameter vector $\BS{\beta}=(\beta_1,\dots,\beta_\nu)^\trp$.  Therefore, the choice of locally $\Phi_k$-optimal weights which yields the minimum value of $\Phi_k(\xi,\BS{\beta})$ over all saturated designs with the same support are  given by the next lemma.
 
\begin{lemma} \label{lem3.0.3.}
Consider a GLM without intercept with $\BS{f}(\BS{x})=(x_1,\dots,x_\nu)^\trp$ on the experimental region $\mathcal{X}$. Denote by $\BS{e}_i$\, for all\, $(1\le i \le \nu)$ the $\nu$-dimensional unit vectors. Let $\BS{x}_i^*=a_i\, \BS{e}_i,\,a_i>0$\, for all\,  $(1\le i \le \nu)$ be   design points  in $\mathcal{X}$ such that the vectors $\BS{f}(\BS{x}_1^*),\ldots,\BS{f}(\BS{x}_\nu^*)$ are linearly independent.   Let $\BS{\beta}=(\beta_1,\dots,\beta_\nu)^\trp$  be a given parameter point.  Let $u_i=u(\BS{x}_i^*,\BS{\beta})$\,for all\, $(1\le i \le \nu)$. For a given positive real vector $\BS{a}=(a_1,\dots,a_\nu)^\trp$ the design $\xi_{\BS{a}}^*$ which achieves the minimum value of $\Phi_k(\xi_{\BS{a}},\BS{\beta})$ over all designs $\xi_{\BS{a}}$ with ${\rm supp}(\xi_{\BS{a}}^*)=\{\BS{x}_1^*,\ldots,\BS{x}_\nu^*\}$ assigns weights 
\begin{equation*}
\omega_i^*=\frac{(a_i^2u_i)^\frac{-k}{k+1}}{\sum\limits_{i=1}^\nu (a_i^2u_i)^\frac{-k}{k+1}}\,\, (1\le i \le \nu)
\end{equation*}
 to the corresponding design points in  $\{\BS{x}_1^*,\dots,\BS{x}_\nu^*\}$.\\[3ex]
For D-optimality ($k=0$),\,\,$\omega_i^*=1/\nu$\,\, $(1\le i\le \nu)$.\\[2ex]
For A-optimality ($k=1$),\,\, $\omega_i^*=\frac{(a_i^2u_{i})^{-1/2}}{\sum\limits_{i=1}^{\nu}(a_i^2u_{i})^{-1/2}}$\,\, $(1\le i\le \nu)$.\\[2ex]
For E-optimality ($k\rightarrow \infty$),\,\, $\omega_i^*=\frac{(a_i^2u_{i})^{-1}}{\sum\limits_{i=1}^{\nu}(a_i^2u_{i})^{-1}}$\,\, $(1\le i\le \nu)$. 
\end{lemma}
 
\begin{proof}
Define  the $\nu\times \nu$ design matrix $\BS{F}=\mathrm{diag}(a_i)_{i=1}^{\nu}$ with the $\nu \times \nu$ weight matrix $\BS{V}=\mathrm{diag}(u_{i}\omega_i)_{i=1}^\nu$. Then we have $
\BS{M}\bigl(\xi_{\BS{a}}, \BS{\beta}\bigr)=\BS{F}^\trp \BS{V}\BS{F}=\mathrm{diag}(a_i^2u_{i}\omega_i)_{i=1}^\nu$ and $\BS{M}^{-k}\bigl(\xi_{\BS{a}}, \BS{\beta}\bigr)=\mathrm{diag}\big((a_i^2u_{i}\omega_i)^{-k}\big)_{i=1}^\nu$ with $\mathrm{tr}\big(\BS{M}^{-k}(\xi_{\BS{a}}, \BS{\beta})\big) =\sum\limits_{i=1}^{\nu}(a_i^2u_{i}\omega_i)^{-k}$.  Note that the  eigenvalues of $\BS{M}^{-k}\bigl(\xi_{\BS{a}}, \BS{\beta}\bigr)$ are given by   $\lambda_i(\xi_{\BS{a}},\BS{\beta})=(a_i^2u_{i}\omega_i)^{-k}$\,\,  $(1\le i \le \nu)$. Thus  the Kiefer $\Phi_k$-criteria can be defined as
\begin{equation}
\Phi_k(\xi_{\BS{a}},\BS{\beta})=\Big(\frac{1}{\nu}\sum\limits_{i=1}^{\nu}(a_i^2u_{i}\omega_i)^{-k}\Big)^{\frac{1}{k}}\,\,\, (0 <k <\infty). \label{eq-3}
\end{equation}
 Now we aim at minimizing $\Phi_k(\xi_{\BS{a}},\BS{\beta})$ such that $\omega_i>0$ and $\sum\limits_{i=1}^{\nu}\omega_i=1$. We write  $\omega_\nu=1-\sum\limits_{i=1}^{\nu-1}\omega_i$ then (\ref{eq-3}) becomes
\begin{equation*}
\Phi_k(\xi_{\BS{a}},\BS{\beta})=\frac{1}{\nu^{1/k}}\Big((a_\nu^{2}u_{\nu})^{-k}(1-\sum\limits_{i=1}^{\nu-1}\omega_i)^{-k}+\sum\limits_{i=1}^{\nu-1}(a_i^2u_{i}\omega_i)^{-k}\Big)^{\frac{1}{k}}. 
\end{equation*}
It is straightforward to see that the  equation  $\frac{\partial \Phi_k(\xi_{\BS{a}},\BS{\beta})}{\partial \omega_i}=0$ is equivalent to 
 \begin{equation*}
 (a_i^2u_i)^{k}\omega^{k+1}_i-(a_\nu^2u_\nu)^{k} (1-\sum\limits_{i=1}^{\nu-1}\omega_i)^{k+1}=0
 \end{equation*}
 which gives  $\omega_i=\Big(a_\nu^2u_\nu/(a_i^2u_i)\Big)^{\frac{k}{k+1}}\omega_\nu$\, ${(1\le i\le \nu-1)}$, thus $\omega_i\, (a_i^2u_{i})^\frac{k}{k+1}=\omega_\nu$\,$(a_\nu^2 u_{\nu})^\frac{k}{k+1}$ ${(1\le i\le \nu-1)}$. This means  $\omega_i\, (a_i^2u_{i})^\frac{k}{k+1}$\,\, ${(1\le i\le \nu)}$ are all equal, i.e.,  $\omega_i\, (a_i^2u_{i})^\frac{k}{k+1}=c$ $(1\le i\le \nu)$, where $c>0$. It implies that  $\omega_i=c\,(a_i^2u_{i})^\frac{-k}{k+1}\,(1\le i\le \nu)$. Due to $\sum\limits_{i=1}^{\nu}\omega_i=1$ we get ${\sum\limits_{i=1}^{\nu}c\,(a_i^2u_{i})^\frac{-k}{k+1}=c\sum\limits_{i=1}^{\nu}(a_i^2u_{i})^\frac{-k}{k+1}=1}$, and thus $c=\bigl(\sum\limits_{i=1}^{\nu}(a_i^2u_{i})^\frac{-k}{k+1}\bigr)^{-1}$. So we finally obtain $\omega_i=(a_i^2u_{i})^\frac{-k}{k+1}/\bigl(\sum\limits_{i=1}^{\nu}(a_i^2u_{i})^\frac{-k}{k+1}\bigr)$\,for all \,$(1\le i\le \nu)$ which are the optimal weights given by the lemma.
\end{proof}
\section{Single-factor model} \label{sec3-2}
In this section we concentrate on the simplest case for which the model is composed by a single factor through the linear predictor 
\[
\eta(x,\BS{\beta})=\BS{f}^\trp(x)\BS{\beta}=\beta_0+\beta_1x\,\,\,\,\mbox{where}\,\,\,\, x \in \mathcal{X}. 
\] 
Let the experimental region is taken to be the continues unit interval $\mathcal{X}=[0,1]$.  We introduce the function 
\begin{equation*}
g(x)=\frac{c}{u(x,\BS{\beta})}
\end{equation*}
with constant $c$.  The function $g(x)$ will be utilized  for the characterization of the optimal designs. Consider the following conditions:\\
($\rmnum{1}$) $u(x,\BS{\beta})$ is positive and twice continuously differentiable.\\
($\rmnum{2}$)   $u(x,\BS{\beta})$ is strictly increasing on $\mathbb{R}$.\\
($\rmnum{3}$) $g^{\prime\prime}(x)$ is  an injective (one-to-one) function.\\
Recently,  Lemma 1 in \Citet{soton346869} showed that under the above conditions  ($\rmnum{1}$)-($\rmnum{3}$) with $g(x)=2/u(x,\BS{\beta})$  a locally D-optimal design on $[0,1]$ is only supported by two points $a$ and $b$ where ${0\le a< b\le 1}$.  In what follows analogous result is presented  for locally optimal designs under various optimality criteria.

\begin{lemma}\label{lem3.1.1.}
Consider model $\BS{f}(x)=(1,x)^\trp$ and experimental region $\mathcal{X}=[0,1]$. Let a parameter point $\BS{\beta}=(\beta_0,\beta)^\trp$ be given. Let conditions ($\rmnum{1}$)-($\rmnum{3}$) be satisfied. Denote by $\BS{A}$ a positive definite matrix and let $c$ be constant. Then if the condition of The General Equivalence Theorem is of the form 
\[
u(x,\BS{\beta})\BS{f}^\trp(x)\BS{A}\BS{f}(x)\le c 
\]
then the support points of a  locally optimal design $\xi^*$ is concentrated on exactly two points $a$ and $b$ where  $0\le a< b\le 1$. 
\end{lemma}
\begin{proof}
Let $\BS{A}=[a_{ij}]_{i,j=1,2}$. Then let $p(x)=\BS{f}^\trp(x)\BS{A}\BS{f}(x)=a_{22}x^2+2a_{12}x+a_{11}$ which is  a polynomial  in $x$ of degree 2 where $x\in \mathcal{X}$. Hence, by  The Equivalence Theorem  $\xi^*$ is locally optimal (at $\BS{\beta}$) if and only if 
\begin{equation*}
p(x)\le g(x) \mbox{ for all } x \in \mathcal{X}. 
\end{equation*}
The above inequality is similar to that obtained in the proof of Lemma 1 in \Citet{soton346869} and thus the rest of our proof is analogous to that. 
\end{proof}
Obviously,  under D-optimality we have  $c=2$ and  $\BS{A}=\BS{M}^{-1}(\xi^*,\BS{\beta})$ whereas, under  A-optimality we have $c=\mathrm{tr}(\BS{M}^{-1}(\xi^*,\BS{\beta}))=\bigl(\sqrt{(a^2+1)/u_b}+\sqrt{(b^2+1)/u_a}\bigr)/(b-a)^2$ where $u_{a}=u(a,\BS{\beta})$ and  ${u_{b}=u(b,\BS{\beta})}$ and  $\BS{A}=\BS{M}^{-2}(\xi^*,\BS{\beta})$.  In general,  under Kiefer $\Phi_k$-criteria we  denote ${c=\mathrm{tr}(\BS{M}^{-k}(\xi^*,\BS{\beta}))}$ and $\BS{A}=\BS{M}^{-k-1}(\xi^*,\BS{\beta})$.  \par

As a consequence of Lemma \ref{lem3.1.1.},  we next provide sufficient conditions for a design whose support is the boundaries of $[0,1]$, i.e.,  $0$ and $1$ to be locally D- or A-optimal on $\mathcal{X}=[0,1]$ at a given $\BS{\beta}$.  Let $q(x)=1/u(x,\BS{\beta})$ and denote $q_0=q^{\frac{1}{2}}(0)$ and $q_1=q^{\frac{1}{2}}(1)$.  
\begin{theorem} \label{theo3.1.1b.}
Consider model $\BS{f}(x)=\bigl(1,x\bigr)^\trp$ and experimental region $\mathcal{X}=[0,1]$. Let a parameter point $\BS{\beta}=(\beta_0,\beta)^\trp$ be given. Let $q(x)$ be positive,  twice continuously differentiable.   Then:\\[.5ex]
($\rmnum{1}$) The unique locally D-optimal design (at $\BS{\beta}$) is the two-point design supported by\\ $0$ and $1$ with equal weights $1/2$ if 
\begin{equation}
q_{0}^{2}+q_{1}^{2} > q^{\prime\prime}(x)/2\, \mbox{ for all } x \in (0,1). \label{conv-D}
\end{equation}
($\rmnum{2}$) The unique locally A-optimal design (at $\BS{\beta}$) is the two-point design supported by\\ $0$ and $1$ with weights
\[
\omega_0^*=\frac{\sqrt{2} q_{0}}{\sqrt{2}q_{0}+q_{1}} \mbox { and }  
\omega_1^*=\frac{q_{1}}{\sqrt{2}q_{0}+q_{1}}, \mbox { respectively}  
\]
if 
\begin{equation}
q_{0}^{2}+q_{1}^{2}+\sqrt{2}q_{0}q_{1} > q^{\prime\prime}(x)/2\, \mbox{ for all } x \in (0,1). \label{conv-A}
\end{equation}
\end{theorem}

\begin{proof}
\underbar{Ad ($i$)} Employing condition (\ref{eq-3.0}) of The Equivalence Theorem for $k=0$ implies that $\xi^*$ is locally D-optimal if and only if 
\begin{equation*}
(1-x)^2q_{0}^{2}+x^2q_{1}^{2}-q(x)\le 0\,\, \forall x \in [0,1]. 
\end{equation*}
Since the support points are $\{0,1\}$  the l.h.s. of the inequality above equals zero at the boundaries of $[0,1]$.  Then it is sufficient to show that the aforementioned l.h.s. is convex on the interior points  $(0,1)$ and this convexity realizes under condition (\ref{conv-D}) asserted in the theorem \\
\underbar{Ad ($ii$)}  This case  can be shown in analogy to case ($\rmnum{1}$) by employing  condition (\ref{eq-3.0}) of The Equivalence Theorem for $k=1$ with $\mathrm{tr}(\BS{M}^{-1}(\xi^*,\BS{\beta}))=(\sqrt{2}q_{0}+q_{1})^2$.
\end{proof}

Now consider the discrete experimental region $\mathcal{X}=\{a,b\},\, a,b \in \mathbb{R}$, i.e., the factor $x$ is binary. In view of Lemma \ref{lem3.1.1.} and  Section \ref{sec3-1} we provide locally D- and A-optimal designs in the next theorem.

\begin{theorem} \label{theo3.1.1.}
Consider model $\BS{f}(x)=\bigl(1,x\bigr)^\trp$ and experimental region $\mathcal{X}=\{a,b\}$
with real numbers $a,b$. Let a parameter point $\BS{\beta}=(\beta_0,\beta)^\trp$ be given. Let $u_{a}=u(a,\BS{\beta})$ and $u_{b}=u(b,\BS{\beta})$.  Then:\\[.5ex]
($\rmnum{1}$) The unique locally D-optimal design (at $\BS{\beta}$) is the two-point design supported by\\ $a$ and $b$ with equal weights $1/2$.\\[.5ex]
($\rmnum{2}$) The unique locally A-optimal design (at $\BS{\beta}$) is the two-point design supported by\\ $a$ and $b$ with weights
\[
\omega_a^*=\frac{u_{a}^{-1/2}\sqrt{1+b^2}}{u_{a}^{-1/2}\sqrt{1+b^2}+u_{b}^{-1/2}\sqrt{1+a^2}}, \ \ 
\omega_b^*=1-\omega_a^*. 
\]
\end{theorem}

The locally D-optimal design given in the previous theorem is independent of the intensities, i.e., it is the same for all generalized linear models. Similar results for Poisson models were indicated in \citet{WANG20062831}.  However, for each setup (or each intensity form) of generalized linear models there is a locally A-optimal design that even varies with parameter values. Since $a$ and $b$ are the only design points there is a locally D- or A-optimal design at any  parameter value in the parameter space of $\BS{\beta}=(\beta_0,\beta_1)^\trp$. \par

\section{Multiple regression models} \label{sec3-3}

In this section we consider a first order model  with multiple factors,
\begin{equation}
\BS{f}(\BS{x})=\bigl(1,\BS{x}^\trp\bigr)^\trp\,\, \mbox{where}\,\,\ \ \BS{x}\in\mathcal{X}. \label{eq3.2}
\end{equation}   
The linear predictor is determined by $\eta(\BS{x},\BS{\beta})=\BS{f}^\trp(\BS{x})\BS{\beta}=\beta_0+\sum_{i=1}^{\nu}\beta_ix_i$ with binary factors. That is a discrete experimental region is considered and has the form $\mathcal{X}=\{0,1\}^\nu,\nu\ge 2$.   We aim at constructing locally D- and A-optimal designs for a given parameter point $\BS{\beta}$  adopting particular analytic solutions. \par

To this end, we firstly begin with  a two-factor model
\begin{equation}
\BS{f}(\BS{x})=\bigl(1,x_1,x_2\bigr)^\trp \mbox{ where }\BS{x}=(x_1,x_2)^\trp \in \mathcal{X}=\{0,1\}^2.   \label{eq3.2.2}
\end{equation} 
The experimental region can be written as  $\mathcal{X}=\{(0,0)^\trp,(1,0)^\trp(0,1)^\trp,(1,1)^\trp\}$. Let us denote the design points by $\BS{x}^*_1=(0,0)^\trp$,  $\BS{x}^*_2=(1,0)^\trp$, $\BS{x}^*_3=(0,1)^\trp$, and $\BS{x}^*_4=(1,1)^\trp$. 
The following results  that are provided in Theorem \ref{theo3.1.2.} and Theorem \ref{theo3.1.3.} under generalized linear model (\ref{eq3.2.2})  are  extensions of the corresponding results that were obtained in \citet{GAFFKE2019} under a gamma model, i.e.,  a GLM with inverse link function $\beta_0+\beta_1x_1+\beta_2x_2=1/\mu$, where $\mu>0$ with intensity $u(\BS{x},\BS{\beta})=(\beta_0+\beta_1x_1+\beta_2x_2)^{-2}$ and the unit cube $[0,1]^2$ as an experimental region. The proofs are analogous to those in the reference.\par

\begin{theorem}\label{theo3.1.2.}
Consider model (\ref{eq3.2.2}) and experimental region $\mathcal{X}=\{0,1\}^2$.
 For a given parameter point $\BS{\beta}=(\beta_0,\beta_1,\beta_2)^\trp$ let $u_i=u(\BS{x}^*_i,\BS{\beta})$\,($1\le i\le4$). Denote by $u_{(1)}\le u_{(2)}\le u_{(3)}\le u_{(4)}$ the intensity values $u_1,u_2,u_3,u_4$ rearranged in ascending order. 
Then:\\[.5ex]
(o) The locally D-optimal design $\xi^*$ (at $\BS{\beta}$) is unique.\\[.5ex]
($\rmnum{1}$) If \ $u_{(1)}^{-1}\,\ge u_{(2)}^{-1}+u_{(3)}^{-1}+u_{(4)}^{-1}$ \ then  
$\xi^*$ is a three-point design supported by the three design points whose intensity values are given by
$u_{(2)}$, $u_{(3)}$, $u_{(4)}$, with equal weights $1/3$.\\[.5ex]
($\rmnum{2}$) If \ $u_{(1)}^{-1}\,< u_{(2)}^{-1}+u_{(3)}^{-1}+u_{(4)}^{-1}$ \ then  
$\xi^*$ is a four-point design supported by the four design points $\BS{x}^*_1,\BS{x}^*_2,\BS{x}^*_3,\BS{x}^*_4$
with weights $\omega_1^*,\omega_2^*,\omega_3^*,\omega_4^*$ which are uniquely determined by the condition
\begin{equation}
 \omega_i^*>0\  (1\le i\le4),\  \sum\limits_{i=1}^4\omega_i^*=1,\ \mbox{and }\ 
u_i\omega_i^*\bigl({\textstyle\frac{1}{3}}-\omega_i^*\bigr)\ \ (1\le i\le4)\ \mbox{are equal.}
\label{eq3.3}
\end{equation}
\end{theorem}

\begin{remark}\ \\[-5ex]                
\begin{enumerate}
\item It is already seen from the optimality conditions asserted in part ($\rmnum{1}$) of   Theorem \ref{theo3.1.2.} that  the design points with highest intensities perform as a support of a locally D-optimal design at a given parameter value. 
\item The optimality condition asserted in part ($\rmnum{2}$) of   Theorem \ref{theo3.1.2.} applies only when the optimality conditions for the three-point (saturated) designs in ($\rmnum{1}$) cannot be  satisfied. 
\end{enumerate}
\end{remark}

Theorem \ref{theo3.1.2.} covers various results in the literature. For examples; see \citet{10.2307/24310039} for  binary responses with several link functions and see \citet{10.1007/978-3-319-00218-7_14} for count data in item response theory.\par

Now consider the case of equally effect sizes; i.e., $\beta_1=\beta_2=\beta$. Next we give explicit formulas for the weights of locally D-optimal four-point designs at parameter points $\BS{\beta}=(\beta_0,\beta_1,\beta_2)$. 

\begin{corollary}\label{theo4.2.1.}
Under the assumptions of Theorem \ref{theo3.1.2.} let the parameter point ${\BS{\beta}=(\beta_0,\beta_1,\beta_2)^\trp}$ be given with $\beta_1=\beta_2=\beta$ such that  assumption ($\rmnum{2}$) of Theorem \ref{theo3.1.2.} is  fulfilled.  Then the locally D-optimal design (at $\BS{\beta}$) 
is supported by the four design points $\BS{x}_1^*,\BS{x}_2^*,\BS{x}_3^*,\BS{x}_4^*$ with weights
\begin{align*}
\omega_{1}^*&=\frac{3}{8}+\frac{1}{4}\Big(1+\frac{u_{1}}{u_{4}}-4\frac{u_{1}}{u_{2}}\Big)^{-1},\\
\omega_{2}^*&=\omega_{3}^*=\frac{1}{2}\Big(4-\frac{u_{2}}{u_{1}}-\frac{u_{2}}{u_{4}}\Big)^{-1},\\
\omega_{4}^*&=\frac{3}{8}+\frac{1}{4}\Big(1+\frac{u_{4}}{u_{1}}-4\frac{u_{4}}{u_{2}}\Big)^{-1}.
 \end{align*}
\end{corollary}
\begin{proof}
Since assumption ($\rmnum{2}$) of Theorem \ref{theo3.1.2.} is  fulfilled by a point $\BS{\beta}$ the design is supported by  all points $\BS{x}^*_1=(0,0)^\trp$,  $\BS{x}^*_2=(1,0)^\trp$, $\BS{x}^*_3=(0,1)^\trp$,  $\BS{x}^*_4=(1,1)^\trp$. Then the optimal weights are obtained in view of  Lemma \ref{lem3.0.2.} where  we have $d_i^2=1\,(1\le i \le 4)$ and $u_2=u_3$. Hence, the results follow.  
\end{proof}

 In analogy to Theorem \ref{theo3.1.2.} we introduce locally A-optimal designs in the next theorem where also  the design points with highest intensities perform as a support of a locally A-optimal design at a given parameter value. 

\begin{theorem} \label{theo3.1.3.} Consider the assumptions and notations of  Theorem \ref{theo3.1.2.}. Denote $q_i=u_i^{-1/2}$\,${(1\le i \le 4)}$.  Then:\\[.5ex]
\hspace*{1.5ex}(o) The locally A-optimal design $\xi^*$ (at $\BS{\beta}$) is unique.\\[-3ex]
\begin{enumerate}[label=(\roman*)]
\item  
If \ $q_{1}^{2}\,\ge q_{2}^{2}+q_{3}^{2}+q_{4}^{2}+ q_{2}q_{3}+2\sqrt{\frac{2}{3}} q_{2}q_{4}+2\sqrt{\frac{2}{3}} q_{3}q_{4} $ \ then 
\[
\xi^*=\left(\begin{array}{ccc} \BS{x}^*_2 & \BS{x}^*_3 &\BS{x}^*_4\\ [.5ex] 
\sqrt{2}q_{2}/c&\sqrt{2}q_{3}/c  &\sqrt{3}q_{4}/c\end{array}\right).
\]
\item  
If \ $q_{2}^{2}\,\ge q_{1}^{2}+q_{3}^{2}+q_{4}^{2}+ q_{1}q_{3}+\sqrt{2}q_{3}q_{4}$ \ then 
\[
\xi^*=\left(\begin{array}{ccc} \BS{x}^*_1 & \BS{x}^*_3 &\BS{x}^*_4\\ [.5ex] 
\sqrt{2}q_{1}/c&\sqrt{2}q_{3}/c  &q_{4}/c\end{array}\right).
\]
\item  
If \ $q_{3}^{2}\,\ge\,q_{1}^{2} + q_{2}^{2} + q_{4}^{2} + q_{1}q_{2} + \sqrt{2}q_{2}q_{4}$ \ then
\[
\xi^*=\left(\begin{array}{ccc} \BS{x}^*_1 & \BS{x}^*_2 &\BS{x}^*_4\\ [.5ex] 
\sqrt{2}q_{1}/c &\sqrt{2}q_{2} &q_{4}/c\end{array}\right).
\]
\item  
If \ $q_{4}^{2}\,\ge\,q_{1}^{2} + q_{2}^{2} + q_{3}^{2} + \frac{2}{\sqrt{3}}q_{1}q_{2} + \frac{2}{\sqrt{3}}q_{1}q_{3}$ \ then
 \[
 \xi^*=\left(\begin{array}{ccc} \BS{x}^*_1 &\BS{x}^*_2 &\BS{x}^*_3\\ [.5ex] 
\sqrt{3}q_{1}/c &q_{2}/c &q_{3}/c\end{array}\right).
\]
\end{enumerate}
For each case ($\rmnum{1}$) -- ($\rmnum{4}$), the constant $c$ appearing in the weights equals the sum of the numerators of the three ratios. If none of the cases ($\rmnum{1}$) -- ($\rmnum{4}$) applies then $\xi^*$ is supported by the four design points $\BS{x}^*_1, \BS{x}^*_2, \BS{x}^*_3, \BS{x}^*_4$.
\end{theorem}

In the following we  consider model (\ref{eq3.2}) for a general number of factors, $\nu\ge2$, and with the experimental region  $\mathcal{X}=\{0,1\}^\nu$. Here, we are interested in providing an extension of  locally D- and A-optimal designs with support $(0,0)^\trp, (1,0)^\trp, (0,1)^\trp$  given in the preceding theorems under a two-factor model.
 
\begin{theorem}\label{theo3.1.4.}            
Consider model (\ref{eq3.2}) with experimental region $\mathcal{X}=\{0,1\}^\nu$, where $\nu\ge2$. Denote the design points by 
\[
\BS{x}_1^*=(0,\dots,0)^\trp,\ \  \BS{x}_2^*=(1,\dots,0)^\trp,\ \ldots,\   \BS{x}_{\nu+1}^*=(0,\dots,1)^\trp. 
\]
For a given parameter point $\BS{\beta}=(\beta_0,\beta_1,\ldots,\beta_\nu)^\trp$ let $u_i=u(\BS{x}_i^*,\BS{\beta})\,(1 \le i \le \nu+1)$. Then the design $\xi^*$   which assigns equal weights   $1/(\nu+1)$ to 
the design points $\BS{x}_i^*$\,\, for all $\,(1 \le i \le \nu+1)$ is locally D-optimal (at $\BS{\beta}$) if and only if 
 \begin{equation}
u_1^{-1}\bigl(1-\textstyle\sum\limits_{j=1}^\nu x_j)^2 + \sum\limits_{i=1}^\nu u_{i+1}^{-1}x_i^2\le
u^{-1}(\BS{x},\BS{\beta})\,\, \mbox{ for all }\,\, \BS{x}\in\bigl\{0,1\bigr\}^\nu.  \label{eqhhh}
\end{equation}
 \end{theorem}

\begin{proof}
Define the $(\nu+1)\times (\nu+1)$ design matrix $\BS{F}=\bigl[\BS{f}(\BS{x}_1^*),\ldots,\BS{f}(\BS{x}_{\nu+1}^*)\bigr]^\trp$, then
\[
\BS{M}(\xi^*,\BS{\beta})=\frac{1}{\nu+1}\BS{F}^\trp\BS{U}\BS{F},\ \mbox{ where }
\BS{U}={\rm diag}\bigl(u_i\bigr)_{i=1}^{\nu+1}. 
\]
We have
\begin{equation}
\BS{F}=\left[\begin{array}{cc}1 & \BS{0}_{1\times\nu}\\ \BS{1}_{\nu\times1} & \BS{I}_\nu
\end{array}\right] ,\ \mbox{ hence }
\BS{F}^{-1}=\left[\begin{array}{cc}1 & \BS{0}_{1\times\nu}\\ -\BS{1}_{\nu\times1} & \BS{I}_\nu\end{array}\right], \label{eq4.228}
 \end{equation}
where $\BS{0}_{1\times\nu}$, $\BS{1}_{\nu\times1}$, and $\BS{I}_\nu$ denote the
$\nu$-dimensional row vector of zeros, the $\nu$-dimensional column vector of ones, and the $\nu\times\nu$ unit matrix, respectively.  So,  by  condition (\ref{eq-3.0}) of The Equivalence Theorem for $k=0$  the design is locally D-optimal if and only if
\begin{eqnarray}
u(\BS{x},\BS{\beta})\,\BS{f}^\trp(\BS{x})\,\BS{M}^{-1}(\xi^*,\BS{\beta})\,\BS{f}(\BS{x})\le \nu+1\,\,\,\forall \BS{x}\in \{0,1\}^\nu. \label{eqh}
\end{eqnarray}
The l.h.s. of (\ref{eqh}) reads as 
\begin{eqnarray*}
&&u(\BS{x},\BS{\beta})\,(\nu+1)\BS{f}^\trp(\BS{x})\,\BS{F}^{-1}\BS{U}^{-1}\bigl(\BS{F}^{-1}\bigr)^\trp  
\BS{f}(\BS{x})=\\&&(\nu+1)u(\BS{x},\BS{\beta})\Bigl(u_1^{-1}\bigl(1-{\textstyle\sum\limits_{j=1}^\nu x_j}\bigr)^2 + \sum\limits_{i=1}^\nu u_{i+1}^{-1}x_i^2\Bigr), 
\end{eqnarray*}
and hence it is obvious that (\ref{eqh}) is equivalent to (\ref{eqhhh}). 
\end{proof}

\begin{remark}
The D-optimal design under a two-factor model with support $(0,0)^\trp$, $(1,0)^\trp$, $(0,1)^\trp$ from Theorem \ref{theo3.1.2.} is covered by Theorem \ref{theo3.1.4.} for $\nu=2$ where condition (\ref{eqhhh}) is equivalent to the inequality
$u_4^{-1}\ge u_1^{-1}+u_2^{-1}+u_3^{-1}$ that is asserted in part ($\rmnum{1}$) of Theorem \ref{theo3.1.2.}. Moreover, Theorem \ref{theo3.1.4.} covers various results in the literature. For example;  \citet{10.2307/24308852} provided for the Poisson model a locally D-optimal saturated design on the continuous experimental region $[0,1]^\nu,\nu \ge 2$ that is  supported by $(0,\dots,0)^\trp$, $(1,\dots,0)^\trp$, $\dots$, $(0,\dots,1)^\trp$  at $\beta_i=-2,(1\le i \le \nu)$.
\end{remark}

 In analogy to Theorem \ref{theo3.1.4.} we introduce locally A-optimal designs in the next theorem.
\begin{theorem}\label{theo3.1.5.}  Consider the assumptions and notations of\, Theorem \ref{theo3.1.4.}. Denote $q_i=u_i^{-1/2}$\, ${(1\le i \le \nu+1)}$. Then the design $\xi^*$ which is supported by $\BS{x}_i^*\,(1\le i \le \nu+1)$ with weights
\[
\omega_1=\sqrt{\nu+1}q_1/c\,\, \mbox{ and } \omega_{i+1}^*=q_{i+1}/c,\,\, i=1,\dots, \nu,\, c=\sqrt{\nu+1}q_1+\sum\limits_{i=2}^{\nu}q_{i}
\]
 is locally A-optimal (at $\BS{\beta}$) if and only if  for all  $\BS{x}=(x_1,\ldots,x_\nu)^\trp\in\bigl\{0,1\bigr\}^\nu$  
\begin{eqnarray}
q_1^{2}\Bigl(1-\sum\limits_{j=1}^{\nu} x_j\Bigr)^2 + \sum\limits_{i=1}^{\nu} q_{i+1}^{2}x_i^2 + 
\frac{2q_1}{\sqrt{\nu+1}}\Bigl(\sum\limits_{j=1}^{\nu}x_j-1\Bigr)\sum\limits_{i=1}^{\nu}q_{i+1}x_i\,\le\,u^{-1}(\BS{x},\BS{\beta}). \label{conA}
\end{eqnarray}
\end{theorem}

\begin{proof}
As in the proof of Theorem \ref{theo3.1.4.} the design matrix $\BS{F}$ and its inverse are given by (\ref{eq4.228}) and we obtain
\[
\BS{C}=\bigl(\BS{F}^{-1}\bigr)^\trp\BS{F}^{-1}=\left[\begin{array}{cc}\nu+1 & -{\bf1}_{1\times\nu}\\
-{\bf1}_{\nu\times 1} & \BS{I}_\nu\end{array}\right].
\]
This yields $\sqrt{c_{11}/u_1}=\sqrt{\nu+1}q_1$ and
$\sqrt{c_{ii}/u_i}=q_i$ for $i=2,\ldots,\nu+1$ according to Section \ref{sec3-1} with $p=\nu+1$. An elementary calculation shows that the weights given in Section \ref{sec3-1} for an A-optimal design coincide with the $\omega_i^*$ ($1\le i\le p$) as stated in the theorem. Now we show that the design $\xi^*$  is locally A-optimal  if and only if  (\ref{conA}) holds. Let $\BS{U}={\rm diag}\bigl(u_1,\ldots,u_p\bigr)$ and  $\BS{\Omega}={\rm diag}\bigl(\omega_1^*,\ldots,\omega_p^*\bigr)$ with the $p \times p$ weight matrix  $\BS{V}=\BS{\Omega}\BS{U}$. Then we have  
\begin{align*}
&\BS{M}(\xi^*,\BS{\beta})=\BS{F}^\trp\BS{V}\BS{F}=\BS{F}^\trp\BS{\Omega}\BS{U}\BS{F},\\
&{\rm tr}\bigl(\BS{M}^{-1}(\xi^*,\BS{\beta})\bigr)={\rm tr}\Bigl(\BS{F}^{-1}\BS{U}^{-1}\BS{\Omega}^{-1}(\BS{F}^{-1})^\trp \Bigr)=c\sum\limits_{i=1}^p\Bigl(\frac{c_{ii}}{u_i}\Bigr)^{1/2}=\,c^2.
\end{align*}
Since $\BS{U}^{-1/2}\BS{\Omega}^{-1}=c\,{\rm diag}\bigl(c_{11}^{-1/2},\ldots,c_{pp}^{-1/2}\bigr)$, we obtain
\[
\BS{M}^{-2}(\xi^*,\BS{\beta})=\BS{F}^{-1}\BS{U}^{-1}\BS{\Omega}^{-1}(\BS{F}^{-1})^\trp\,
\BS{F}^{-1}\BS{U}^{-1}\BS{\Omega}^{-1}(\BS{F}^{-1})^\trp= c^2\,\BS{F}^{-1}\BS{U}^{-1/2}\BS{C}^*\BS{U}^{-1/2}(\BS{F}^{-1})^\trp
\]
 where $\BS{C}^*={\rm diag}\bigl(c_{11}^{-1/2},\ldots,c_{pp}^{-1/2}\bigr)\,\BS{C}\,{\rm diag}\bigl(c_{11}^{-1/2},\ldots,c_{pp}^{-1/2}\bigr)$
\ 
So, together with condition (\ref{eq-3.0}) of The General Equivalence Theorem for $k=1$ the design $\xi^*$ is locally A-optimal (at $\BS{\beta}$) if and only if
 \begin{eqnarray}
&&\Bigl(\BS{U}^{-1/2}(\BS{F}^{-1})^\trp \BS{f}(\BS{x})\Bigr)^\trp\BS{C}^*\Bigl(\BS{U}^{-1/2}(\BS{F}^{-1})^\trp \BS{f}(\BS{x})\Bigr)\,
\le u^{-1}(\BS{x},\BS{\beta})\ \ \forall\ \BS{x}\in\{0,1\}^\nu \label{eqA}
\end{eqnarray}
 Straightforward calculation shows that condition (\ref{conA}) that provides  a characterization of local A-optimality of $\xi^*$  is equivalent to  condition (\ref{eqA}).
\end{proof}

\begin{remark}
Theorem \ref{theo3.1.5.} with $\nu=2$ covers the result stated in  case ($\rmnum{4}$)  of Theorem \ref{theo3.1.3.}. It can  be checked that,
with the notations of Theorem \ref{theo3.1.3.}, the inequality $ q_{4}^{2}\,> q_{1}^{2}+q_{2}^{2}+q_{3}^{2}+\frac{2}{\sqrt{3}}q_{1}q_{2}+\frac{2}{\sqrt{3}}q_{1}q_{3}$  is equivalent to assumption (\ref{conA}) of Theorem \ref{theo3.1.5.} for $\nu=2$.   
\end{remark}

\section{Model without intercept}\label{sec3-5}
 
 In this section we consider  GLMs (\ref{eq2.2}) having a linear predictor without intercept, i.e., the components  $f_j\neq 1$  for all  ($1\le j\le p$) and thus $f_j(\BS{0})=\BS{0}$ for all ($1\le j\le p$). Precisely, we focus on a first order model   
 \begin{equation*}
 \BS{f}(\BS{x})=(x_1,\dots,x_\nu)^\trp\,\,\mbox{ where }\,\,\,\BS{x}\in \mathcal{X}.
 \end{equation*}
Next locally optimal designs will be derived under  Kiefer $\Phi_k$-criteria and thus,  the results  implicitly cover the D-, A- and E-optimal designs.  We provide necessary and sufficient conditions for constructing  $\Phi_k$-optimal designs on a general experimental region  $\mathcal{X}$. The support points are located at the boundaries  of $\mathcal{X}$ and the optimal weights are obtained according to Lemma \ref{lem3.0.3.}.

\begin{theorem} \label{theo5.0.1}
Consider the experimental region $\mathcal{X}$. Given a vector $\BS{a}=(a_1,\dots,a_\nu)^\trp$ where $a_i\in \mathbb{R}$,\, $a_i>0\,\,(1\le i\le \nu)$. Let $\BS{x}_i^*=a_i\BS{e}_i\,\,(1\le i\le \nu)$ denote the design points that belong to $\mathcal{X}$. For a given  parameter point $\BS{\beta}$  denote $u_i=u(\BS{x}_{i}^*,\BS{\beta})$ \,$(1\le i\le \nu)$. Let $\xi_{\BS{a}}^*$ be the saturated design whose support is $\BS{x}_i^*$\, $(1\le i\le \nu)$ with the corresponding  weights 
\[
\omega_i^*=\frac{(a_i^2u_i)^\frac{-k}{k+1}}{\sum\limits_{i=1}^\nu (a_i^2u_i)^\frac{-k}{k+1}}\,\,(1\le i \le \nu). 
\]
Then $\xi_{\BS{a}}^*$ is locally $\Phi_k$-optimal (at $\BS{\beta}$)  if and only if  
\begin{eqnarray}
u(\BS{x},\BS{\beta})\sum\limits_{i=1}^{\nu}u_{i}^{-1}a_{i}^{-2}x_i^2 \leq 1\,\,\,\,\mbox{ for all}\,\,\BS{x} = (x_1, \dots, x_\nu)^\trp \in \mathcal{X}.  \label{eq3.12}
\end{eqnarray}
 \end{theorem}

\begin{proof} 
Define  the $\nu\times \nu$ design matrix $\BS{F}=\mathrm{diag}(a_i)_{i=1}^\nu$ with the $\nu\times \nu$  weight matrix 
\[
\BS{V}=\mathrm{diag}(u_{i}\omega^*_i)_{i=1}^\nu= \Big(\sum\limits_{i=1}^{\nu}(a_{i}^{2}u_{i})^\frac{-k}{k+1}\Big)^{-1}\mathrm{diag}\Big((a_{i}^{-2k}u_{i})^\frac{1}{k+1}\Big)_{i=1}^{\nu}.
\]
Then we have  
\begin{align*}
&\BS{M}\bigl(\xi_{\BS{a}}^*, \BS{\beta}\bigr)=\BS{F}^\trp\BS{V}\BS{F}=  \Big(\sum\limits_{i=1}^{\nu}(a_{i}^{2}u_{i})^\frac{-k}{k+1}\Big)^{-1}\mathrm{diag}\Big((a_{i}^{2}u_{i})^\frac{1}{k+1}\Big)_{i=1}^{\nu},\\
&\BS{M}^{-k-1}\bigl(\xi_{\BS{a}}^*, \BS{\beta}\bigr)=\Big(\sum\limits_{i=1}^{\nu}(a_{i}^{2}u_{i})^\frac{-k}{k+1}\Big)^{k+1}\mathrm{diag}\Big(a_{i}^{-2}u_{i}^{-1}\Big)_{i=1}^{\nu}, \mbox{ and }\\
&\mathrm{tr}\Big(\BS{M}^{-k}\bigl(\xi_{\BS{a}}^*, \BS{\beta}\bigr)\Big)=\Big (\sum\limits_{i=1}^{\nu}(a_{i}^{2}u_{i})^\frac{-k}{k+1}\Big)^{k+1}.
\end{align*}
Adopting these formulas simplifies the l.h.s. of condition (\ref{eq-3.0}) of The Equivalence Theorem to $u(\BS{x},\BS{\beta})\Big(\sum\limits_{i=1}^{\nu}(a_{i}^{2}u_{i})^\frac{-k}{k+1}\Big)^{k+1}\sum\limits_{i=1}^{\nu}u_{i}^{-1}a_{i}^{-2}x_i^2$ which is hence, bounded by $\Big(\sum\limits_{i=1}^{\nu}(a_{i}^{2}u_{i})^\frac{-k}{k+1}\Big)^{k+1}$ if and only if  condition (\ref{eq3.12}) holds true.
 \end{proof}
 
 In particular,  Theorem \ref{theo5.0.1} states that for a given parameter point  $\BS{\beta}$ the locally D-optimal design ($k=0$)  has  wights $\omega_i^*=1/\nu\,\, (1\le i\le \nu)$ and the  locally A-optimal design ($k=1$) has weights $\omega_i^*=\frac{(a_i^2u_{i})^{-1/2}}{\sum\limits_{i=1}^{\nu}(a_i^2u_{i})^{-1/2}}$\,\, $(1\le i\le \nu)$ whereas  the  locally E-optimal design ($k\rightarrow \infty$) has weights $\omega_i^*=\frac{(a_i^2u_{i})^{-1}}{\sum\limits_{i=1}^{\nu}(a_i^2u_{i})^{-1}}$\,\, $(1\le i\le \nu)$. \par
 
Theorem \ref{theo5.0.1} might be applicable for a wide class of GLMs on appropriate experimental regions. Consider a non-intercept gamma model, i.e., $\BS{x}^\trp\BS{\beta}=1/\mu$ with $\mu>0$ and  intensity $u(\BS{x},\BS{\beta})=(\BS{x}^\trp\BS{\beta})^{-2}$. Let the experimental region is given by $\mathcal{X}=[0,\infty)\setminus\{\BS{0}\}$. Due to  the positivity assumption of gamma models, i.e., $\mu>0$ the parameter point $\BS{\beta}$ must satisfy the condition  $\BS{x}^\trp\BS{\beta}>0$ for all $\BS{x}\in \mathcal{X}$. Therefore, the parameter space is determined by $\BS{\beta}\in (0,\infty)^\nu$, i.e.,  $\beta_i>0$ for all ($1 \le i \le \nu$). The next corollary is immediate.

\begin{corollary} \label{theo5.0.2.}
Consider  a non-intercept gamma model  with $\BS{f}(\BS{x})=\BS{x}$ on the experimental region ${\mathcal{X}=[0,\infty)^\nu\setminus\{\BS{0}\}}$ and intensity $u(\BS{x},\BS{\beta})=(\BS{x}^\trp\BS{\beta})^{-2}$. Given a vector $\BS{a}=(a_1,\dots,a_\nu)^\trp$ where $a_i\in \mathbb{R}$,\, $a_i>0\,\,(1\le i\le \nu)$. Let ${\BS{x}_i^*=a_i\BS{e}_i}$\, for all $i=1,\dots, \nu$ denote the design points belong to $\mathcal{X}$. For a given parameter point $\BS{\beta}\in (0,\infty)^\nu$ let $\xi_{\BS{a}}^*$ be the saturated design whose support is $\BS{x}_i^*$\, $(1\le i\le \nu)$ with the corresponding weights 
\[
\omega_i^*=\frac{\beta_i^\frac{2k}{k+1}}{\sum\limits_{i=1}^\nu \beta_i^\frac{2k}{k+1}}\,\,\,(1\le i \le \nu).
\]
Then $\xi_{\BS{a}}^*$ is locally $\Phi_k$-optimal (at $\BS{\beta}$).
\end{corollary}

\begin{proof} 
The corollary covers the result of Theorem \ref{theo5.0.1} under a gamma model.  For a given $\BS{\beta} \in (0,\infty)^\nu$ let $u_i=u(\BS{x}_i^*,\BS{\beta})\,\,(1\le i \le \nu)$. Thus $u_i=(a_i\beta_i)^{-2}\,\,(1\le i \le \nu)$.  Then  condition (\ref{eq3.12}) of Theorem \ref{theo5.0.1} is equivalent to $-2\sum\limits_{i<j=1}^{\nu}\beta_i\beta_jx_ix_j\le 0$ for all $\BS{x} \in \mathcal{X}$. Since ${\beta_i>0, x_i>0\,\,\, (1\le i \le \nu)}$ the condition holds  true for any $\BS{x} \in \mathcal{X}$ at any given $\BS{\beta} \in (0,\infty)^\nu$. 
 \end{proof}

Corollary \ref{theo5.0.2.} covers Theorem 3.1 in \citet{2019arXiv190409232I} which provided locally D- and A-optimal designs for non-intercept gamma models. For a Poisson model, i.e.,  $\BS{x}^\trp\BS{\beta}=\log (\mu)$ with intensity  $u(\BS{x},\BS{\beta})=\exp\big(\BS{x}^\trp\BS{\beta}\big)$ and experimental region $\mathcal{X}=\{0,1\}^\nu, \nu\ge 2$  let us restrict to the case of  $a_i=1\,\,(1 \le i \le \nu)$, i.e., the design points are the unit vectors $\BS{e}_i$\,$(1\le i \le \nu)$. As a result, condition (\ref{eq3.12}) is simplified as presented in the following corollary. 

\begin{corollary} \label{cor3.2.1.}
Consider a non-intercept Poisson model with $\BS{f}(\BS{x})=\BS{x}$ on the experimental region $\mathcal{X}=\{0,1\}^\nu,\,\,\nu\ge2$ and  intensity $u(\BS{x},\BS{\beta})=\exp\big(\BS{x}^\trp\BS{\beta}\big)$.  For a given parameter point $\BS{\beta}=(\beta_1,\dots,\beta_\nu)^\trp$ define $\lambda_i=\exp(\beta_i)\,\,(1\le i\le \nu)$ and denote by $\lambda_{[1]}\ge\lambda_{[2]}\ge\dots\ge\lambda_{[\nu]}$ the descending order of $\lambda_{1},\lambda_{2},\dots,\lambda_{\nu}$. Let $\xi_{\BS{a}}^*$ be the saturated design  supported by the unit vectors $\BS{e}_i$\,$(1\le i \le \nu)$  with weights $\omega_i^*=\frac{\lambda_i^\frac{-k}{k+1}}{\sum\limits_{i=1}^\nu \lambda_i^\frac{-k}{k+1}}\,\,(1\le i \le \nu)$. Then $\xi_{\BS{a}}^*$ is locally $\Phi_k$-optimal (at $\BS{\beta}$) if and only if  
\begin{equation}
\lambda_{[1]}+\lambda_{[2]}\le1. \label{eq3.13}
\end{equation}
\end{corollary}
\begin{proof} The corollary covers the result of Theorem \ref{theo5.0.1} under a Poisson model with  intensity $u(\BS{x},\BS{\beta})=\exp\big(\BS{x}^\trp\BS{\beta}\big)$ and $a_i=1\,\,(1 \le i \le \nu)$. So condition (\ref{eq3.12}) reduces to 
 \begin{equation}
   \exp(\sum\limits_{i=1}^{\nu}\beta_ix_i)\sum\limits_{i=1}^{\nu}\exp(-\beta_i)x_i^2\le1\,\,\, \forall \BS{x}\in \mathcal{X}. \label{eq3.14}
\end{equation}
For any $\BS{x}=(x_1,\dots,x_\nu) \in \{0,1\}^\nu,\nu\ge 2$ define the index set  $S\subseteq\{1,\ldots,\nu\}$   such that 
  $x_i=1$\, if \,$i\in S$\, and \,$x_i=0$ else.  So for \,$\BS{x}$\, described by \,$S\subseteq\{1,\ldots,\nu\}$\, and\, $s=\#S$,\, if  $s=0$ (i.e., $S=\emptyset$)  then the l.h.s. of (\ref{eq3.14}) is zero. If\,$s=1$, \,inequality  (\ref{eq3.14})  becomes an equality.  However,  the l.h.s. of (\ref{eq3.14})  is equal to   $\exp(\sum\limits_{i\in S}\beta_i)\sum\limits_{i\in S}\exp(-\beta_i)$ which thus rewrites  as $ \prod\limits_{i\in S}\lambda_i\sum\limits_{i\in S}\lambda_i^{-1}$ or equivalently  as $\sum\limits_{i\in S}\prod\limits_{j\in S\setminus\{i\}}\lambda_j$. By the  the descending order $\lambda_{[1]}\ge\lambda_{[2]}\ge\dots\ge\lambda_{[\nu]}$ of  $\lambda_{1},\lambda_{2},\dots,\lambda_{\nu}$ we obtain  for all subsets $S\subseteq\{1,\ldots,\nu\}$ of same sizes $s\ge 2$, 
  \[
  \sum\limits_{i=1}^{s}\lambda_{[i]}^{-1}\prod\limits_{i=1}^{s}\lambda_{[i]} =\sum\limits_{i=1}^s\prod\limits_{ i\neq j=1}^{s}\lambda_{[j]}\ge \sum\limits_{i\in S}\prod\limits_{j\in S\setminus\{i\}}\lambda_j.
  \]
    Denote $T_s=\sum\limits_{i=1}^{s}\lambda_{[i]}^{-1}\prod\limits_{i=1}^{s}\lambda_{[i]}$. Hence, inequality (\ref{eq3.14}) is equivalent to $T_s \le 1$ for all $s=2,\dots,\nu$. Then it is sufficient to show that 
\begin{align*}
\lambda_{[1]}+\lambda_{[2]}&\le1\,\, \iff T_s \le 1\,\,\, \forall s=2,\dots,\nu.
\end{align*}
For ``$\Longleftarrow$'', let $s=2$ then $T_2=\lambda_{[1]}+\lambda_{[2]}\le 1$. For ``$\Longrightarrow$'', firstly, note that $T_2=\lambda_{[1]}+\lambda_{[2]}$ thus $T_s\le 1$ is true for $s=2$.  Now assume $T_s\le 1$ is true  for some $s=q<\nu$, i.e.,  $T_q\le 1$ and we want to show that it is true for $s=q+1$.
We can write 
\begin{align*}
T_{q+1}&=\Big(\sum\limits_{i=1}^{q}\lambda_{[i]}^{-1}+\lambda_{[q+1]}^{-1}\Big)\Big(\prod\limits_{i=1}^{q}\lambda_{[i]}\Big)\lambda_{[q+1]}\\
&=T_q\lambda_{[q+1]}+\prod\limits_{i=1}^{q}\lambda_{[i]}=T_q\lambda_{[q+1]}+T_q\Big(\sum\limits_{i=1}^{q}\lambda_{[i]}^{-1}\Big)^{-1}=T_q\Big(\lambda_{[q+1]}+\Big(\sum\limits_{i=1}^{q}\lambda_{[i]}^{-1}\Big)^{-1}\Big)\\
&\mbox{ since } \Big(\sum\limits_{i=1}^{q}\lambda_{[i]}^{-1}\Big)^{-1} \le \frac{1}{q}\lambda_{[1]} \mbox{ and } \lambda_{[q+1]}+\frac{1}{q}\lambda_{[1]}\le T_2=\lambda_{[1]}+\lambda_{[2]}\le 1 \mbox{ we have } \\
T_{q+1}&\le T_q\Big(\lambda_{[q+1]}+\frac{1}{q}\lambda_{[1]}\Big)\le T_q T_2\le 1.
\end{align*}
\end{proof}

\begin{remark}
One can slightly highlight on $\Phi_k$-optimality  under a non-intercept linear model with ${\BS{f}(\BS{x})=(x_1,\dots,x_\nu)^\trp}$ on the continuous experimental region ${\mathcal{X}=[0,1]^\nu,\,\nu\ge 2}$. Here, $u(\BS{x},\BS{\beta})=1$ for all $\BS{x}\in \mathcal{X}$ so the information matrices in a linear model are independent of $\BS{\beta}$. Note that Theorem \ref{theo5.0.1} does not cover a non-intercept linear model on $\mathcal{X}$ since condition (\ref{eq3.12}) does not hold true for $\nu\ge 2$. However,  the l.h.s. of  condition (\ref{eq-3.0}) of The Equivalence Theorem  under linear models, i.e. $u(\BS{x},\BS{\beta})=1$, is strictly convex and of course it attains its maximum at some vertices of  $\mathcal{X}$. Thus the support  of any $\Phi_k$(or D, A, E)-optimal design is a subset of $\{0,1\}^\nu$. As a result, in particular for D- and A-optimality,  one might apply the  results of Theorem 3.1 in \citet{doi:10.1080/02331888808802124} which were obtained under  linear models on $\{0,1\}^\nu$.
\begin{itemize}
\item  
 For odd numbers of factors $\nu=2q+1,\,\,q\in \mathbb{N}$, the equally weighted designs $\xi^*$  supported by all $\BS{x}^*=(x_1,\dots,x_\nu)\in \{0,1\}^\nu$ such that $\sum\limits_{i=1}^{\nu}x_i=q+1$ is either  D- or A-optimal.
\item  
 For even numbers of  factors $\nu=2q,\,\,q\in \mathbb{N}$, the equally weighted design $\xi^*$  supported by all $\BS{x}^*=(x_1,\dots,x_\nu)\in \{0,1\}^\nu$ such that $\sum\limits_{i=1}^{\nu}x_i=q$ or  $\sum\limits_{i=1}^{\nu}x_i=q+1$ is D-optimal. Moreover, the design $\xi^*$ which assigns equal weights to all points  $\BS{x}^*=(x_1,\dots,x_\nu)\in \{0,1\}^\nu$  such that $\sum\limits_{i=1}^{\nu}x_i=q$ is A-optimal. 
 \end{itemize}
\end{remark}


\begin{thebibliography}{38}
\expandafter\ifx\csname natexlab\endcsname\relax\def\natexlab#1{#1}\fi
\providecommand{\url}[1]{\texttt{#1}}
\providecommand{\href}[2]{#2}
\providecommand{\path}[1]{#1}
\providecommand{\DOIprefix}{doi:}
\providecommand{\ArXivprefix}{arXiv:}
\providecommand{\URLprefix}{URL: }
\providecommand{\Pubmedprefix}{pmid:}
\providecommand{\doi}[1]{\href{http://dx.doi.org/#1}{\path{#1}}}
\providecommand{\Pubmed}[1]{\href{pmid:#1}{\path{#1}}}
\providecommand{\bibinfo}[2]{#2}
\ifx\xfnm\relax \def\xfnm[#1]{\unskip,\space#1}\fi
\bibitem[{Abdelbasit and Plackett(1983)}]{10.2307/2287114}
\bibinfo{author}{Abdelbasit, K.M.}, \bibinfo{author}{Plackett, R.L.},
  \bibinfo{year}{1983}.
\newblock \bibinfo{title}{Experimental design for binary data}.
\newblock \bibinfo{journal}{Journal of the American Statistical Association}
  \bibinfo{volume}{78}, \bibinfo{pages}{90--98}.
\bibitem[{Atkinson et~al.(2007)Atkinson, Donev and
  Tobias}]{atkinson2007optimum}
\bibinfo{author}{Atkinson, A.}, \bibinfo{author}{Donev, A.},
  \bibinfo{author}{Tobias, R.}, \bibinfo{year}{2007}.
\newblock \bibinfo{title}{Optimum Experimental Designs, with SAS}.
  volume~\bibinfo{volume}{34}.
\newblock \bibinfo{publisher}{Oxford University Press, Oxford}.
\bibitem[{Atkinson and Haines(1996)}]{ATKINSON1996437}
\bibinfo{author}{Atkinson, A.}, \bibinfo{author}{Haines, L.},
  \bibinfo{year}{1996}.
\newblock \bibinfo{title}{14 designs for nonlinear and generalized linear
  models}, in: \bibinfo{editor}{Ghosh, S.}, \bibinfo{editor}{Rao, C.} (Eds.),
  \bibinfo{booktitle}{Design and Analysis of Experiments}.
  \bibinfo{publisher}{Elsevier, Amsterdam}. volume~\bibinfo{volume}{13} of
  \textit{\bibinfo{series}{Handbook of Statistics}}, pp.
  \bibinfo{pages}{437--475}.
\bibitem[{Atkinson and Woods(2015)}]{atkinson2015designs}
\bibinfo{author}{Atkinson, A.C.}, \bibinfo{author}{Woods, D.C.},
  \bibinfo{year}{2015}.
\newblock \bibinfo{title}{Designs for generalized linear models}, in:
  \bibinfo{editor}{Angela~Dean, Max~Morris, J.S.}, \bibinfo{editor}{Bingha, D.}
  (Eds.), \bibinfo{booktitle}{Handbook of Design and Analysis of Experiments}.
  \bibinfo{publisher}{Chapman \& Hall/CRC Press, Boca Raton}, pp.
  \bibinfo{pages}{471--514}.
\bibitem[{Biedermann et~al.(2006)Biedermann, Dette and
  Zhu}]{biedermann2006optimal}
\bibinfo{author}{Biedermann, S.}, \bibinfo{author}{Dette, H.},
  \bibinfo{author}{Zhu, W.}, \bibinfo{year}{2006}.
\newblock \bibinfo{title}{Optimal designs for dose--response models with
  restricted design spaces}.
\newblock \bibinfo{journal}{Journal of the American Statistical Association}
  \bibinfo{volume}{101}, \bibinfo{pages}{747--759}.
\bibitem[{Chernoff(1953)}]{chernoff1953}
\bibinfo{author}{Chernoff, H.}, \bibinfo{year}{1953}.
\newblock \bibinfo{title}{Locally optimal designs for estimating parameters}.
\newblock \bibinfo{journal}{Ann. Math. Statist.} \bibinfo{volume}{24},
  \bibinfo{pages}{586--602}.
\bibitem[{Dobson and Barnett(2018)}]{dobson2018introduction}
\bibinfo{author}{Dobson, A.J.}, \bibinfo{author}{Barnett, A.G.},
  \bibinfo{year}{2018}.
\newblock \bibinfo{title}{An Introduction to Generalized Linear Models}.
\newblock \bibinfo{edition}{Fourth edition} ed., \bibinfo{publisher}{CRC press,
  Boca Raton}.
\bibitem[{Fahrmeir and Kaufmann(1985)}]{fahrmeir1985}
\bibinfo{author}{Fahrmeir, L.}, \bibinfo{author}{Kaufmann, H.},
  \bibinfo{year}{1985}.
\newblock \bibinfo{title}{Consistency and asymptotic normality of the maximum
  likelihood estimator in generalized linear models}.
\newblock \bibinfo{journal}{Ann. Statist.} \bibinfo{volume}{13},
  \bibinfo{pages}{342--368}.
\bibitem[{Fedorov and Leonov(2013)}]{fedorov2013optimal}
\bibinfo{author}{Fedorov, V.V.}, \bibinfo{author}{Leonov, S.L.},
  \bibinfo{year}{2013}.
\newblock \bibinfo{title}{Optimal Design for Nonlinear Response Models}.
\newblock \bibinfo{publisher}{CRC Press, Boca Raton}.
\bibitem[{Ford et~al.(1992)Ford, Torsney and Wu}]{10.2307/2346142}
\bibinfo{author}{Ford, I.}, \bibinfo{author}{Torsney, B.}, \bibinfo{author}{Wu,
  C.F.J.}, \bibinfo{year}{1992}.
\newblock \bibinfo{title}{The use of a canonical form in the construction of
  locally optimal designs for non-linear problems}.
\newblock \bibinfo{journal}{Journal of the Royal Statistical Society. Series B
  (Methodological)} \bibinfo{volume}{54}, \bibinfo{pages}{569--583}.
\bibitem[{Fox(2015)}]{fox2015applied}
\bibinfo{author}{Fox, J.}, \bibinfo{year}{2015}.
\newblock \bibinfo{title}{Applied Regression Analysis and Generalized Linear
  Models}.
\newblock \bibinfo{edition}{Third edition} ed., \bibinfo{publisher}{SAGE
  Publications, Los Angeles}.
\bibitem[{Gaffke et~al.(2019)Gaffke, Idais and Schwabe}]{GAFFKE2019}
\bibinfo{author}{Gaffke, N.}, \bibinfo{author}{Idais, O.},
  \bibinfo{author}{Schwabe, R.}, \bibinfo{year}{2019}.
\newblock \bibinfo{title}{Locally optimal designs for gamma models}.
\newblock \bibinfo{journal}{Journal of Statistical Planning and Inference}
  \DOIprefix\doi{https://doi.org/10.1016/j.jspi.2019.04.002}.
\bibitem[{Goldburd et~al.(2016)Goldburd, Khare and
  Tevet}]{goldburd2016generalized}
\bibinfo{author}{Goldburd, M.}, \bibinfo{author}{Khare, A.},
  \bibinfo{author}{Tevet, D.}, \bibinfo{year}{2016}.
\newblock \bibinfo{title}{Generalized linear models for insurance rating}, in:
  \bibinfo{booktitle}{Casualty Actuarial Society, Virginia}.
\bibitem[{Graßhoff et~al.(2007)Graßhoff, Großmann, Holling and
  Schwabe}]{GRAHOFF20073882}
\bibinfo{author}{Graßhoff, U.}, \bibinfo{author}{Großmann, H.},
  \bibinfo{author}{Holling, H.}, \bibinfo{author}{Schwabe, R.},
  \bibinfo{year}{2007}.
\newblock \bibinfo{title}{Design optimality in multi-factor generalized linear
  models in the presence of an unrestricted quantitative factor}.
\newblock \bibinfo{journal}{Journal of Statistical Planning and Inference}
  \bibinfo{volume}{137}, \bibinfo{pages}{3882--3893}.
\bibitem[{Gra{\ss}hoff et~al.(2013)Gra{\ss}hoff, Holling and
  Schwabe}]{10.1007/978-3-319-00218-7_14}
\bibinfo{author}{Gra{\ss}hoff, U.}, \bibinfo{author}{Holling, H.},
  \bibinfo{author}{Schwabe, R.}, \bibinfo{year}{2013}.
\newblock \bibinfo{title}{Optimal design for count data with binary predictors
  in item response theory}, in: \bibinfo{editor}{Ucinski, D.},
  \bibinfo{editor}{Atkinson, A.C.}, \bibinfo{editor}{Patan, M.} (Eds.),
  \bibinfo{booktitle}{mODa 10-Advances in Model-Oriented Design and Analysis},
  \bibinfo{publisher}{Springer International Publishing, Heidelberg}. pp.
  \bibinfo{pages}{117--124}.
\bibitem[{Gra{\ss}hoff et~al.(2015)Gra{\ss}hoff, Holling and
  Schwabe}]{10.1007/978-3-319-13881-7_9}
\bibinfo{author}{Gra{\ss}hoff, U.}, \bibinfo{author}{Holling, H.},
  \bibinfo{author}{Schwabe, R.}, \bibinfo{year}{2015}.
\newblock \bibinfo{title}{Poisson model with three binary predictors: when are
  saturated designs optimal?}, in: \bibinfo{editor}{Steland, A.},
  \bibinfo{editor}{Rafaj{\l}owicz, E.}, \bibinfo{editor}{Szajowski, K.} (Eds.),
  \bibinfo{booktitle}{Stochastic Models, Statistics and Their Applications},
  \bibinfo{publisher}{Springer International Publishing, Cham}. pp.
  \bibinfo{pages}{75--81}.
\bibitem[{{Gra{\ss}hoff} et~al.(2018){Gra{\ss}hoff}, {Holling} and
  {Schwabe}}]{2018arXiv181003893G}
\bibinfo{author}{{Gra{\ss}hoff}, U.}, \bibinfo{author}{{Holling}, H.},
  \bibinfo{author}{{Schwabe}, R.}, \bibinfo{year}{2018}.
\newblock \bibinfo{title}{{D-Optimal design for the Rasch counts model with
  multiple binary predictors}}.
\newblock \bibinfo{journal}{ArXiv e-prints}
  \href{http://arxiv.org/abs/1810.03893}{{\tt arXiv:1810.03893}}.
\bibitem[{Huda and Mukerjee(1988)}]{doi:10.1080/02331888808802124}
\bibinfo{author}{Huda, S.}, \bibinfo{author}{Mukerjee, R.},
  \bibinfo{year}{1988}.
\newblock \bibinfo{title}{Optimal weighing designs: approximate theory}.
\newblock \bibinfo{journal}{Statistics} \bibinfo{volume}{19},
  \bibinfo{pages}{513--517}.
\bibitem[{{Idais} and {Schwabe}(2019)}]{2019arXiv190409232I}
\bibinfo{author}{{Idais}, O.}, \bibinfo{author}{{Schwabe}, R.},
  \bibinfo{year}{2019}.
\newblock \bibinfo{title}{{Analytic solutions for locally optimal designs for
  gamma models having linear predictor without intercept}}.
\newblock \bibinfo{journal}{arXiv e-prints}
  \href{http://arxiv.org/abs/1904.09232}{{\tt arXiv:1904.09232}}.
\bibitem[{Kabera et~al.(2015)Kabera, Haines and
  Ndlovu}]{doi:10.1080/02331888.2014.937342}
\bibinfo{author}{Kabera, G.M.}, \bibinfo{author}{Haines, L.M.},
  \bibinfo{author}{Ndlovu, P.}, \bibinfo{year}{2015}.
\newblock \bibinfo{title}{The analytic construction of d-optimal designs for
  the two-variable binary logistic regression model without interaction}.
\newblock \bibinfo{journal}{Statistics} \bibinfo{volume}{49},
  \bibinfo{pages}{1169--1186}.
\bibitem[{Khuri et~al.(2006)Khuri, Mukherjee, Sinha and Ghosh}]{khuri2006}
\bibinfo{author}{Khuri, A.I.}, \bibinfo{author}{Mukherjee, B.},
  \bibinfo{author}{Sinha, B.K.}, \bibinfo{author}{Ghosh, M.},
  \bibinfo{year}{2006}.
\newblock \bibinfo{title}{Design issues for generalized linear models: A
  review}.
\newblock \bibinfo{journal}{Statist. Sci.} \bibinfo{volume}{21},
  \bibinfo{pages}{376--399}.
\bibitem[{Kiefer(1975)}]{doi:10.1093biomet62.2.277}
\bibinfo{author}{Kiefer, J.}, \bibinfo{year}{1975}.
\newblock \bibinfo{title}{Optimal design: Variation in structure and
  performance under change of criterion}.
\newblock \bibinfo{journal}{Biometrika} \bibinfo{volume}{62},
  \bibinfo{pages}{277--288}.
\bibitem[{Kiefer and Wolfowitz(1960)}]{kiefer_wolfowitz_1960}
\bibinfo{author}{Kiefer, J.}, \bibinfo{author}{Wolfowitz, J.},
  \bibinfo{year}{1960}.
\newblock \bibinfo{title}{The equivalence of two extremum problems}.
\newblock \bibinfo{journal}{Canadian Journal of Mathematics}
  \bibinfo{volume}{12}, \bibinfo{pages}{363--366}.
\bibitem[{Konstantinou et~al.(2014)Konstantinou, Biedermann and
  Kimber}]{soton346869}
\bibinfo{author}{Konstantinou, M.}, \bibinfo{author}{Biedermann, S.},
  \bibinfo{author}{Kimber, A.}, \bibinfo{year}{2014}.
\newblock \bibinfo{title}{Optimal designs for two-parameter nonlinear models
  with application to survival models}.
\newblock \bibinfo{journal}{Statistica Sinica} \bibinfo{volume}{24},
  \bibinfo{pages}{415--428}.
\bibitem[{Mathew and Sinha(2001)}]{MATHEW2001295}
\bibinfo{author}{Mathew, T.}, \bibinfo{author}{Sinha, B.K.},
  \bibinfo{year}{2001}.
\newblock \bibinfo{title}{Optimal designs for binary data under logistic
  regression}.
\newblock \bibinfo{journal}{Journal of Statistical Planning and Inference}
  \bibinfo{volume}{93}, \bibinfo{pages}{295--307}.
\bibitem[{McCullagh and Nelder(1989)}]{mccullagh1989generalized}
\bibinfo{author}{McCullagh, P.}, \bibinfo{author}{Nelder, J.},
  \bibinfo{year}{1989}.
\newblock \bibinfo{title}{Generalized Linear Models}.
\newblock \bibinfo{edition}{Second edition} ed., \bibinfo{publisher}{Chapman \&
  Hall, London}.
\bibitem[{Myers and Montgomery(1997)}]{doi:10.1080/00224065.1997.11979769}
\bibinfo{author}{Myers, R.H.}, \bibinfo{author}{Montgomery, D.C.},
  \bibinfo{year}{1997}.
\newblock \bibinfo{title}{A tutorial on generalized linear models}.
\newblock \bibinfo{journal}{Journal of Quality Technology}
  \bibinfo{volume}{29}, \bibinfo{pages}{274--291}.
\bibitem[{Nelder and Wedderburn(1972)}]{10.2307/2344614}
\bibinfo{author}{Nelder, J.A.}, \bibinfo{author}{Wedderburn, R.W.M.},
  \bibinfo{year}{1972}.
\newblock \bibinfo{title}{Generalized linear models}.
\newblock \bibinfo{journal}{Journal of the Royal Statistical Society. Series A
  (General)} \bibinfo{volume}{135}, \bibinfo{pages}{370--384}.
\bibitem[{Pukelsheim(1993)}]{pukelsheim2006optimal}
\bibinfo{author}{Pukelsheim, F.}, \bibinfo{year}{1993}.
\newblock \bibinfo{title}{Optimal Design of Experiments}.
\newblock \bibinfo{publisher}{Wiley, New York}.
\bibitem[{Russell et~al.(2009)Russell, Woods, Lewis and
  Eccleston}]{10.2307/24308852}
\bibinfo{author}{Russell, K.G.}, \bibinfo{author}{Woods, D.C.},
  \bibinfo{author}{Lewis, S.M.}, \bibinfo{author}{Eccleston, J.A.},
  \bibinfo{year}{2009}.
\newblock \bibinfo{title}{D-optimal designs for poisson regression models}.
\newblock \bibinfo{journal}{Statistica Sinica} \bibinfo{volume}{19},
  \bibinfo{pages}{721--730}.
\bibitem[{Schmidt and Schwabe(2017)}]{schmidt2017optimal}
\bibinfo{author}{Schmidt, D.}, \bibinfo{author}{Schwabe, R.},
  \bibinfo{year}{2017}.
\newblock \bibinfo{title}{Optimal design for multiple regression with
  information driven by the linear predictor}.
\newblock \bibinfo{journal}{Statistica Sinica} \bibinfo{volume}{27},
  \bibinfo{pages}{1371--1384}.
\bibitem[{Silvey(1980)}]{silvey1980optimal}
\bibinfo{author}{Silvey, S.D.}, \bibinfo{year}{1980}.
\newblock \bibinfo{title}{Optimal Design}.
\newblock \bibinfo{publisher}{Chapman \& Hall, London}.
\bibitem[{Tong et~al.(2014)Tong, Volkmer and Yang}]{tong2014}
\bibinfo{author}{Tong, L.}, \bibinfo{author}{Volkmer, H.W.},
  \bibinfo{author}{Yang, J.}, \bibinfo{year}{2014}.
\newblock \bibinfo{title}{Analytic solutions for d-optimal factorial designs
  under generalized linear models}.
\newblock \bibinfo{journal}{Electron. J. Statist.} \bibinfo{volume}{8},
  \bibinfo{pages}{1322--1344}.
\bibitem[{Walker and Duncan(1967)}]{10.2307/2333860}
\bibinfo{author}{Walker, S.H.}, \bibinfo{author}{Duncan, D.B.},
  \bibinfo{year}{1967}.
\newblock \bibinfo{title}{Estimation of the probability of an event as a
  function of several independent variables}.
\newblock \bibinfo{journal}{Biometrika} \bibinfo{volume}{54},
  \bibinfo{pages}{167--179}.
\bibitem[{Wang et~al.(2006)Wang, Myers, Smith and Ye}]{WANG20062831}
\bibinfo{author}{Wang, Y.}, \bibinfo{author}{Myers, R.H.},
  \bibinfo{author}{Smith, E.P.}, \bibinfo{author}{Ye, K.},
  \bibinfo{year}{2006}.
\newblock \bibinfo{title}{D-optimal designs for poisson regression models}.
\newblock \bibinfo{journal}{Journal of Statistical Planning and Inference}
  \bibinfo{volume}{136}, \bibinfo{pages}{2831--2845}.
\bibitem[{Yang et~al.(2012)Yang, Mandal and Majumdar}]{10.2307/24310039}
\bibinfo{author}{Yang, J.}, \bibinfo{author}{Mandal, A.},
  \bibinfo{author}{Majumdar, D.}, \bibinfo{year}{2012}.
\newblock \bibinfo{title}{Optimal designs for two-level fractional experiments
  with binary response}.
\newblock \bibinfo{journal}{Statistica Sinica} \bibinfo{volume}{22},
  \bibinfo{pages}{885--907}.
\bibitem[{Yang(2008)}]{YANG2008624}
\bibinfo{author}{Yang, M.}, \bibinfo{year}{2008}.
\newblock \bibinfo{title}{A-optimal designs for generalized linear models with
  two parameters}.
\newblock \bibinfo{journal}{Journal of Statistical Planning and Inference}
  \bibinfo{volume}{138}, \bibinfo{pages}{624--641}.
\bibitem[{Yang and Stufken(2009)}]{yang2009}
\bibinfo{author}{Yang, M.}, \bibinfo{author}{Stufken, J.},
  \bibinfo{year}{2009}.
\newblock \bibinfo{title}{Support points of locally optimal designs for
  nonlinear models with two parameters}.
\newblock \bibinfo{journal}{The Annals of Statistics} \bibinfo{volume}{37},
  \bibinfo{pages}{518--541}.

\end{thebibliography}
\end{document}